\documentclass[12pt, twoside]{article}
\usepackage{amsfonts,mathrsfs,amsmath,amssymb,amsthm,upref,amscd}
\usepackage{xcolor}
\usepackage[all,cmtip]{xy}
\usepackage{setspace}
\usepackage{enumerate}
\usepackage[titletoc]{appendix}
\usepackage{times}
\usepackage{cite}
\usepackage{tikz}
\usepackage[colorinlistoftodos]{todonotes}
\usetikzlibrary{arrows}
\numberwithin{equation}{section}

\makeatletter 
\@mparswitchfalse
\makeatother
\normalmarginpar 

\def\titlerunning#1{\gdef\titrun{#1}}
\makeatletter
\def\author#1{\gdef\autrun{\def\and{\unskip, }#1}\gdef\@author{#1}}
\def\address#1{{\def\and{\\\hspace*{18pt}}\renewcommand{\thefootnote}{}%
		\footnote {#1}}%
	\markboth{\autrun}{\titrun}}
\makeatother
\def\email#1{e-mail: #1}
\def\subjclass#1{{\renewcommand{\thefootnote}{}%
		\footnote{\emph{Mathematics Subject Classification (2010):} #1}}}
\def\keywords#1{\par\medskip
	\noindent\textbf{Keywords.} #1}

\allowdisplaybreaks

\theoremstyle{plain}
\newtheorem{Thm}{Theorem}[section]
\newtheorem{Lem}[Thm]{Lemma}
\newtheorem{Cor}[Thm]{Corollary}
\newtheorem{Prop}[Thm]{Proposition}
\newtheorem*{Thm*}{Theorem}
\newtheorem*{claim*}{Claim}
\newtheorem*{Cor*}{Corollary}

\newtheorem*{Ques*}{Question}

\newtheorem*{Prob*}{Problem}
\newtheorem*{OProb*}{Open Problem}
\theoremstyle{definition}
\newtheorem{Def}[Thm]{Definition}
\newtheorem*{Def*}{Definition}
\newtheorem{Rem}[Thm]{Remark}
\newtheorem{Ex}[Thm]{Example}

\DeclareMathOperator{\real}{Re}

\DeclareMathOperator{\spann}{span}

\DeclareMathOperator{\vol}{vol}

\DeclareMathOperator{\Spin}{Spin}
\DeclareMathOperator{\SO}{SO}

\DeclareMathOperator{\scal}{Scal}
\DeclareMathOperator{\ricc}{Ric}
\DeclareMathOperator{\rank}{rank}

\newcommand{\equ}{equation}
\newcommand{\C}{\mathbb{C}}
\newcommand{\N}{\mathbb{N}}
\newcommand{\R}{\mathbb{R}}
\newcommand{\Z}{\mathbb{Z}}
\newcommand{\bH}{\mathbb{H}}

\newcommand\ca{\mathcal{A}}

\newcommand\cc{\mathcal{C}}

\newcommand\ce{\mathcal{E}}

\newcommand\cl{\mathcal{L}}
\newcommand\cm{\mathcal{M}}

\newcommand\co{\mathcal{O}}

\newcommand\cq{\mathcal{Q}}

\newcommand\cs{\mathcal{S}}

\newcommand\cu{\mathcal{U}}

\newcommand{\inp}[2]{\left\langle#1,#2\right\rangle}
\newcommand{\normm}[1]{{\left\vert\kern-0.25ex\left\vert\kern-0.25ex\left\vert #1 
		\right\vert\kern-0.25ex\right\vert\kern-0.25ex\right\vert}}

\def\mbs{\mathbb{S}}
\def\msa{\mathscr{A}}

\def\msf{\mathscr{F}}

\def\msl{\mathscr{L}}

\def\msp{\mathscr{P}}

\def\mfm{\mathfrak{M}}
\def\mff{\mathfrak{F}}

\def\ig{\textit{g}}

\def\ov{\overline}

\def\pa {\partial}
\def\op{\oplus}

\def\De{\Delta}
\def\ka{\kappa}
\def\al{\alpha}
\def\bt{\beta}

\def\de{\delta}
\def\Ga{\Gamma}
\def\ga{\gamma}

\def\lm{\lambda}

\def\om{\omega}

\def\sa{\sigma}
\def\vr{\varepsilon}
\def\va{\varphi}

\begin{document}
	
	\titlerunning{Conformal deformation of a Riemannian metric via Einstein-Dirac equation}
	
	
	\title{Conformal deformation of a Riemannian metric via an Einstein-Dirac parabolic flow}
	
	\author{Yannick Sire \quad  Tian Xu}
	
	\date{}
	
	\maketitle
	
	\address{
		Y. Sire: Department of Mathematics, Johns Hopkins University, 3400 N. Charles Street, Baltimore, Maryland 21218;
		\email{ysire1@jhu.edu}
		\and
		T. Xu: Department of Mathematics, Zhejiang Normal University, Jinhua, Zhejiang, 321004, China; \email{xutian@amss.ac.cn} 
	}
	
	
	\begin{abstract}
		
	We introduce a new parabolic flow deforming any Riemannian metric on a spin manifold by following a constrained gradient flow of the total scalar curvature. This flow is built out of the well-known Dirac-Einstein functional. We prove local well-posedness of smooth solutions. The present contribution is the first installment of more general program on the Einstein-Dirac problem.

		\vspace{.5cm}
	\end{abstract}
	
	\tableofcontents
	
	\section{Introduction}
	
	The conformally invariant variational problems that arise in geometry and theoretical physics often exhibit rich and subtle mathematical structures. Uncovering and utilizing these structures to explore new phenomena can lead to some of the most challenging and interesting problems in geometric analysis. One of the most famous examples in this context is the Yamabe problem on a closed Riemannian manifold $(M,\ig)$, $\dim M\geq3$, which aims at finding a metric $\tilde\ig$ in the given conformal class $[\ig]:=\{e^{2f}\ig:\, f\in C^\infty(M)\}$ such that the scalar curvature $\scal_{\tilde\ig}$ is constant (this problem has been settled by a series of works of H. Yamabe, N. Trudinger, T. Aubin and R. Schoen, see for instance \cite{Aubin, Schoen, Trudinger, Yamabe} and \cite{LeeParker} for a good overview).  The major steps in the resolution of the Yamabe problem have deep links with variational methods, elliptic partial differential equations, and general relativity. 
	
    Let us mention here that there exists a parabolic proof of the Yamabe problem, which is somehow more desirable, as it focuses on an evolution equation that deforms any Riemannian metric conformally to a constant scalar curvature one. In his seminal paper \cite{Hamilton}, R.S. Hamilton has introduced the so-called Yamabe flow which is given by the following equation
    \[
    \pa_t\ig(t)=-(\scal_{\ig(t)}-\ov\scal_{\ig(t)})\ig(t), \quad \ig(0)=\ig_0,
    \]
    where $\ig_0$ stands for the initial metric and $\ov\scal_{\ig(t)}=\int_M\scal_{\ig(t)}d\vol_{\ig(t)}/\int_Md\vol_{\ig(t)}$ is the average scalar curvature of the evolved metric $\ig(t)$. And this gave rise to an extensive literature, see for instance \cite{Brendle05, Brendle07, Chow, SchSt, Ye}.
    
    In the context of spin geometry, spinor bundles are an important tool in differential geometry as well as in mathematical physics, where they model fermionic particles. A nonlinear coupling of the gravitational field with an eigenspinor of the Dirac operator via the energy-momentum tensor gives rise to an interesting problem, which is to find a Riemannian structure admitting an eigenspinor such that its energy-momentum tensor satisfies the Einstein equation. More precisely, let $M$ be a closed manifold of dimension $m\geq3$ and equipped with a {\it topological spin structure} on $M$ (the term ``topological" refers to the fact that the definition does not depend on a metric; however, we remark here that it does need a differentiable structure on $M$, not only a topological one, see \cite{BGM} for more details on this). For a fixed real constant $\lm\neq0$, the starting point is the action functional
    \begin{\equ}\label{ED-F0}
    W(\ig, \psi):=\int_M \Big[ \scal_{\ig}+(D_\ig\psi,\psi)_\ig-\lm|\psi|_\ig^2 \Big] d\vol_\ig
    \end{\equ}
    describing the interaction of a spinor field $\psi$ and a gravitational metric $\ig$,  whose volume element is denoted by $d\vol_\ig$. We use $D_\ig$ for the Dirac operator derived from the metric $\ig$ and $(\cdot,\cdot)_\ig$ for the hermitian scalar product on the spinor bundle $\mbs(M,\ig)$. The Euler-Lagrange equations are the Einstein and the Dirac equation
    \begin{\equ}\label{ED0}
    \ricc_\ig-\frac12\scal_\ig \ig = \frac14T_{(\ig,\psi)}, \quad D_\ig\psi = \lm\psi,
    \end{\equ}
    where the energy-momentum tensor $T_{(\ig,\psi)}$ is given by the formula
    \[
    T_{(\ig,\psi)}(X,Y)=\real(X\cdot\nabla_Y^\ig\psi+Y\cdot\nabla_X^\ig\psi,\psi)_\ig \quad \forall X, Y\in TM
    \]
    with ``$\cdot$" being the Clifford multiplication and $\nabla^\ig$ being the metric connection on $\mbs(M,\ig)$. The coupled system \eqref{ED0}, henceforth called the {\it Einstein-Dirac equation}, has been considered in physics for a long time in dimension $4$ and Lorentzian signature, as it describes the interaction for a particle of spin $\frac12$ with the gravitation field (cf. \cite{BW1957, FSY}). See also  \cite{KF00, Friedrich-ED00, Belgun} for some studies on the Riemannian spin  $3$-manifolds.
    
    By restricting the functional \eqref{ED-F0} on a given conformal class $[\ig]$ and using the conformal covariance of the Dirac operator (cf. \cite{Hijazi, Hitchin})
    \begin{\equ}\label{conformal-Dirac}
    D_{\ig}\psi=e^{\frac{m+1}2f}D_{e^{2f}\ig}\big( e^{-\frac{m-1}2f}\psi \big),
    \end{\equ}
    we obtain a functional $\Phi$ given by
    \[
    \Phi(u,\psi):=W\big(u^{\frac4{m-2}}\ig, u^{-\frac{m-1}{m-2}}\psi\big)=\int_M\Big[ uL_\ig u + (D_\ig \psi,\psi)_\ig-\lm  u^{\frac2{m-2}}|\psi|_\ig^2 \Big]d\vol_{\ig}
    \]
    for a positive function $u$ and a spinor field $\psi$,
    where $L_\ig=-\frac{4(m-1)}{m-2}\De_\ig+\scal_{\ig}$ is the conformal Yamabe operator. The corresponding Euler-Lagrange equations are
    \begin{\equ}\label{ED-1}
    	\left\{
    	\aligned
    	&L_\ig u = \frac{\lm}{m-2}|\psi|_{\ig}^2 u^{\frac{4-m}{m-2}}, \\
    	&D_\ig\psi = \lm u^{\frac2{m-2}}\psi ,
    	\endaligned
    	\right. \quad u>0
    \end{\equ}
    which is conformally invariant. Such a coupled system has drawn much attention in recent years, see for instance \cite{MM19, BM21} and references therein, where the regularity theorems,  Fredholm theorems and bubbling phenomenon have been investigated in dimension $3$. 
    
    Let us point out here that one can use the simple scaling $v=\mu u$ and $\va=\tau\psi$, with $\mu,\tau>0$, to transform \eqref{ED-1} equivalently to
    \[
    \left\{
    \aligned
    &L_\ig v = \frac{\lm}{m-2}\tau^{-2}\mu^{\frac{2m-6}{m-2}}|\va|_{\ig}^2 v^{\frac{4-m}{m-2}}, \\
    &D_\ig\va = \lm \mu^{-\frac2{m-2}}v^{\frac2{m-2}}\va ,
    \endaligned
    \right. \quad v>0.
    \]
    And hence, by suitably choosing $\mu$ and $\tau$, we can see that the specific value of $\lm$ in Eq. \eqref{ED-1} is rather irrelevant. On the other hand, since the Yamabe operator is involved, one can expect that the sign of $\lm$ would play a role in the analysis of Eq. \eqref{ED-1}, say for instance $\lm>0$ (or $\lm<0$) represents the case of positive (or negative) scalar curvature. The existence problem for the Einstein-Dirac equation in the form of \eqref{ED-1} remains largely unexplored. One of the goals of the present work is specifically to design a parabolic approach to the existence issue for \eqref{ED-1}. Additionally, we want to propose and develop a parabolic approach to the conformal deformation of Riemannian metrics based on the previous Einstein-Dirac equation. Since the first equation in  \eqref{ED-1} is elliptic, the naive approach to make it parabolic by flowing both equations does not provide a suitable parabolic PDE. Indeed, some standard computations to evolve geometric quantities under the flow do not lead to anything useful. On the other hand, the  second equation in \eqref{ED-1} can be viewed as a generalized eigenvalue problem for the Dirac operator $D_\ig$. Using  the link with the Yamabe flow which is the gradient flow of the total scalar curvature, we couple the evolution of scalar curvature via Yamabe flow, {\sl under the generalized eigenvalue equation for the spinor interpreted  as a constraint for the flow.}  It is then natural to relate the flow with a conformal evolution of the Riemannian metric since the Yamabe operator and the scalar curvature are involved, just as in the Yamabe problem. And thus, we introduce the following flow for the Einstein-Dirac equation (which will be formally introduced in Subsection \ref{subsec:set-up-of-the-problem}):
    \begin{\equ}\label{the-flow0}
    \frac{\pa u^{\frac{m+2}{m-2}}}{\pa t}=-\frac{m+2}{m-2}\bigg[ L_{\ig}u-\bigg(\frac{\int_M uL_{\ig}u\,d\vol_{\ig}}{\int_M u^{\frac2{m-2}}|\psi_u|_\ig^2\,d\vol_{\ig}}\bigg) |\psi_u|_{\ig}^2 u^{\frac{4-m}{m-2}} \bigg].
    \end{\equ}
    where $\psi_u$ stands for the generalized Dirac-eigenspinor along the flow (the corresponding eigenvalue, which is not directly appearing in \eqref{the-flow0},  will be denoted by $\lm_u$). The initial data consist of a positive function $u(0)=u_0$ on $M$ and an eigenpair $(\lm_u(0),\psi_u(0))=(\lm_0,\psi_0)$. We shall consider this problem as a model for a parabolic approach to other problems in geometric analysis that involve a nonlocal term.  
    
    In what follows, let $(M,\ig,\sa)$ be an $m$-dimensional closed spin manifold, $m\geq3$, with a fixed spin structure $\sa:P_{\Spin}(M)\to P_{\SO}(M)$ and a fixed smooth Riemannian metric $\ig$. Endow the space $\cm$ of Riemannian metrics over $M$ with the $C^\infty$-topology. In order to provide a precise description of our result, we recall  a ``rigidity" condition introduced in \cite{Canzani} for the Dirac operator.
    
    \begin{Def}
    	An eigenspace of the Dirac operator $D_\ig$ is said to be a rigid eigenspace if it has dimension greater or equal than two, and for any two eigenspinors $\psi$, $\va$ with $\int_M|\psi|_\ig^2d\vol_{\ig}=\int_M|\va|_\ig^2d\vol_{\ig}=1$ then
    	\[
    	|\psi(x)|_\ig=|\va(x)|_\ig \quad \forall x\in M.
    	\]
    \end{Def}
    
    Denote  $\cc_+^\infty$ the set of all positive smooth functions on $M$ and equip it with $C^\infty$-topology. For each $u_0\in\cc_+^\infty$, let $(\lm_0,\psi_0)$ with $\lm_0\neq0$ be a nontrivial eigenpair of the generalized eigenvalue problem
    \begin{\equ}\label{eigen-equ}
    D_\ig\psi=\lm u_0^{\frac2{m-2}}\psi \quad \text{for } (\lm,\psi)\in \R\times\Ga(\mbs(M,\ig)).
    \end{\equ}

    Our main result  is the following   
    \begin{Thm}\label{main-thm}
    Assume $D_\ig$ has no rigid eigenspaces. For a generic choice of $u_0\in \cc_+^\infty$, 
    the problem \eqref{the-flow0} with initial data $(u_0,\lm_0,\psi_0)$ admits a unique smooth solution in the time interval $[0,T)$ for some $T>0$.
    \end{Thm}
    
    Here, by ``generic" we mean the set of functions $u_0\in\cc_+^\infty$ for which Theorem \ref{main-thm} holds is a residual set of $\cc_+^\infty$.
    
    The previous theorem provides only the local well-posedness for \eqref{the-flow0} . We would like to emphasize that global well-posedness does not follow from standard techniques due to the special structure of our PDE.  It is important to notice that the problem we consider is of parabolic-elliptic nature. Of course, our Theorem \ref{main-thm} implies, by standard parabolic arguments, global well-posedness for {\sl small} data.  More importantly, the {\sl convergence} of the flow is a difficult problem. We will address it in a forthcoming work but we would like to point out that if convergence holds at infinity we expect that the limit would solve
     \begin{\equ}\label{ED-2}
   			\left\{
   			\aligned
   			& L_\ig u =  \frac{\int_M uL_\ig u  \, d\vol_{\ig}}{\int_M u^{\frac2{m-2}}|\psi|_\ig^2 \, d\vol_\ig}|\psi|_\ig^2 u^{\frac{4-m}{m-2}}, \\[0.3em]
   			& D_\ig\psi=\frac{\int_M(D_\ig\psi,\psi)_\ig  \, d\vol_\ig}{\int_M u^{\frac2{m-2}}|\psi|_\ig^2 \, d\vol_\ig} u^{\frac2{m-2}}\psi,
   			\endaligned
   			\right. \quad u>0.
   		\end{\equ}
   		and a direct computation shows that there is a one-to-one correspondence between solutions of \eqref{ED-1} and \eqref{ED-2}.
   		
   		Of course, one would like to get rid of the ``rigidity" assumption in Theorem \ref{main-thm}. This will be the case in dimension three since the quaternionic rank of the spinor bundle is simple, that is $\rank_{\bH}\mbs(M)=1$ when $\dim M=3$. Hence, as a specific case, we have
   		
   		\begin{Thm}\label{main-thm2}
          For a generic smooth positive function $u_0$ on a closed spin $3$-manifold $(M,\ig,\sa)$, and an arbitrary choice of $(\lm_0,\psi_0)$ satisfying
          \begin{\equ}\label{eigne-equ-3}
          D_\ig\psi_0=\lm_0 u_0^2\psi_0
          \end{\equ}
          with $\lm_0\neq0$ and $\psi_0\neq0$. There exist $T>0$ and a unique smooth triple $(u,\lm_u,\psi_u):[0,T)\to \cc_+^\infty\times\R\times\Ga(\mbs(M,\ig))$ so that $u(0)=u_0$, $\lm_u(0)=\lm_0$, $\psi_u(0)=\psi_0$ and
          \[
          D_\ig\psi_u(t)=\lm_u(t) u(t)^2\psi_u(t) \quad \text{for all time } t\in[0,T).
          \]
         Furthermore, $(u,\psi_u)$ evolves through the following flow equation
          \[
          \frac{\pa u}{\pa t}=-\bigg[ L_{\ig}u-\bigg(\frac{\int_M uL_{\ig}u\,d\vol_{\ig}}{\int_M u^{2}|\psi_u|_\ig^2\,d\vol_{\ig}}\bigg) |\psi_u|_{\ig}^2 u \bigg]u^{-4}.
          \]
   		\end{Thm}
   		
   		Here we would like to provide an explanation for the generic selection of the initial value $u_0\in\cc_+^\infty$. In fact, our results need to be established upon the basis that $\lm_u(t)\neq0$ always remains a simple eigenvalue of the eigen-problem \eqref{eigen-equ} or \eqref{eigne-equ-3}. This is because, if at some time $t$, $\lm_u(t)$ is a multiple eigenvalue, it is unclear whether the corresponding eigenspinor can smoothly evolve along the Einstein-Dirac flow. Even in the simplest case of matrix operators, we can easily find examples where the initial matrix, under a smooth perturbation, has a multiple eigenvalue and its corresponding eigenvector that do not continuously depend on the perturbation. Therefore, this limitation currently restricts us to only obtaining short-time existence of the Einstein-Dirac flow.
   
  This paper is organized as follows. In Section 2, we will briefly introduce some notations and the Sobolev spaces that will be used later. In Section 3, we will describe the continuous dependence properties of the solutions to the eigenvalue problem of the Dirac operator under continuous changes in the Riemannian metric. In Section 4, we will establish existence and uniqueness for solutions of a linear parabolic equation with non-local term, which will serve as a fundamental tool for studying the flow problem mentioned earlier. In Section 5, we will prove Theorems \ref{main-thm} and \ref{main-thm2}, which provide the short time existence for the Einstein-Dirac parabolic flow \eqref{the-flow0}.

	\section{Basic notations}\label{sec:basic notations}
	
		
In this section, we shall briefly recall some notations and properties of conformal differential operators involved, and provide the definitions of the Sobolev spaces that we will be using.

Let $(M,\ig)$ be a closed Riemannian $m$-manifold, $m\geq3$, we denote the conformal Laplacian (also known as the Yamabe operator) acting on scalar functions by
\[
L_\ig u:= -\frac{4(m-1)}{m-2}\De_\ig u + \scal_\ig u
\]
where $\De_\ig$ is the Laplace-Beltrami operator and $\scal_{\ig}$ is the scalar curvature. The conformal covariance of $L_\ig$ reads as
\[
L_\ig u= e^{\frac{m+2}{2}f}L_{e^{2f}\ig}\big( e^{-\frac{m-2}2f}u \big).
\]
For an integer $k$, and $u:M\to\R$ smooth enough, we denote by $\nabla_\ig^k u$ the $k$th covariant derivative of $u$ and $|\nabla_\ig^ku|$ the norm of $\nabla_\ig^ku$ defined in a local chart by
\[
|\nabla_\ig^ku|=\ig^{i_1 j_1}\cdots\ig^{i_kj_k}(\nabla_\ig^ku)_{i_1\dots i_k}(\nabla_\ig^k u)_{j_1\dots j_k}
\]
Recall that $(\nabla_\ig u)_i=\pa_i u$ while $(\nabla_\ig^2)_{ij}=\pa_{i}\pa_j u-\Ga_{ij}^k\pa_k u$, where $\Ga_{ij}^k$ is the Christoffel symbols. In the sequel, we will denote $C^k(M)=\{u:M\to\R \,|\, \nabla_\ig^j u \text{ is continuous for } j=0,1,\dots,k\}$ and $\cc_+^k:=\{u\in C^k(M) \,|\, u>0\}$. The notation $u\in C^\infty(M)$ (respectively $u\in \cc_+^\infty$) will be understood as $u\in C^k(M)$ (respectively $u\in \cc_+^k$) for all $k\in\N$. We also denote $H^k(M)$ the Sobolev spaces on $M$ equipped with the norm
\[
\|u\|_{H^k(M)}=\sqrt{\sum_{j=0}^k\int_M|\nabla_\ig^ju|^2d\vol_{\ig}}.
\]
The inner product $\inp{\cdot}{\cdot}_{H^k(M)}$ associated to $\|\cdot\|_{H^k(M)}$ is given by
\[
\inp{u}{v}_{H^k(M)}=\sum_{j=0}^k\int_M\inp{\nabla_\ig^j u}{\nabla_\ig^j v} d\vol_{\ig}
\]
where, in the above expression, $\inp{\cdot}{\cdot}$ is the inner product on covariant tensor fields associated to $\ig$. Note here that, by the Sobolev embedding theorems, one can look at $H^k(M)$ as subspaces of $L^p(M)$ for $1\leq p\leq \frac{2m}{m-2k}$. And the standard $L^p$-norm will be simply denoted by $|\cdot|_p$.

When $(M,\ig)$ is equipped additionally with a spin structure, we shall consider $\mbs(M)$ the canonical spinor bundle over $M$ (see \cite{Friedrich00}), whose sections are called spinors on $M$. This bundle is endowed with a natural Clifford multiplication $\mfm$, a hermitian metric $(\cdot,\cdot)_\ig$ and a natural metric connection $\nabla^{\mbs(M)}$. The Dirac operator $D_\ig$ acts on spinors
\[
D_\ig: \Ga(\mbs(M))\to\Ga(\mbs(M))
\]
is defined as the composition $\mfm\circ\nabla^{\mbs(M)}$ in the sense that
\[
\nabla^{\mbs(M)}:\Ga(\mbs(M))\to \Ga(T^*M\otimes \mbs(M)),
\]
\[
\mfm:\Ga(TM\otimes \mbs(M))\to \Ga(\mbs(M)),
\]
where $T^*M\simeq TM$ have been identified by means of the metric $\ig$. We also have a conformal covariance here:
\[
D_{\ig}\psi=e^{\frac{m+1}2f}D_{e^{2f}\ig}\big( e^{-\frac{m-1}2f}\psi \big),
\]
see for instance \cite{Hijazi, Hitchin}. Let us remind here that the Dirac operator $D_\ig$ on a closed spin manifold is essentially self-adjoint in $L^2(M,\mbs(M))$, has compact resolvent and there exists a complete $L^2$-orthonormal basis of eigenspinors $\{\psi_i\}_{i\in\Z}$:
\[
D_\ig\psi_i=\lm_i\psi_i,
\]
and the eigenvalues $\{\lm_i\}_{i\in\Z}$ are of finite multiplicities and are unbounded on both sides of the real line (i.e. $|\lm_i|\to+\infty$ as $|i|\to+\infty$). For a spinor $\psi\in L^2(M,\mbs(M))$, it has a representation 
\[
\psi=\sum_{i\in \Z}\al_i\psi_i
\]
with $\al_i\in\C$ so that $\sum_{i}|\al_i|^2<+\infty$. And for $s>0$, we can define the unbounded operator $|D_\ig|^s: L^2(M,\mbs(M))\to L^2(M,\mbs(M))$ by
\[
|D_\ig|^s\psi=\sum_{i\in\Z}\al_i|\lm_i|^s\psi_i.
\]
Let us denote $H^s(M,\mbs(M))$ the domain of $|D_\ig|^s$, namely $\psi=\sum_i\al_i\psi_i\in H^s(M,\mbs(M))$ if and only if
\[
\sum_{i\in\Z}|\al_i|^2|\lm_i|^{2s}<+\infty.
\]
Then $H^s(M,\mbs(M))$ coincides with the standard Sobolev space $W^{s,2}(M,\mbs(M))$ (see \cite{Ammann}) and can be endowed with the inner product
\[
\inp{\psi}{\va}_s=\real\int_M\big(|D_\ig|^s\psi,|D_\ig|^s\va\big)_\ig d\vol_\ig + \int_M(\psi,\va)_\ig d\vol_{\ig}.
\]

\section{Perturbation for a spinorial eigenvalue problem}\label{sec:eigenvalues}

In this section, we are mainly interested in a generalized eigenvalue perturbation problem for the Dirac operator. More precisely, we intend to find the eigenvalues and eigenspinors of the equation
\begin{\equ}\label{Dirac-eigenvalue-problem}
D_\ig\psi=\lm u^{\frac2{m-2}}\psi, 
\end{\equ} 
that is perturbed from $D_\ig\psi_0=\lm_0 u_0^{\frac2{m-2}}\psi_0$ with known  eigenvalue and eigenspinor. Here some basic conventions are in order, such as we will always assume that $u_0$ and $u$ are positive functions, i.e., $u_0,u\in \cc_+^k$ for some $k\in\N\cup\{\infty\}$. 
%

In what follows, let us collect all possible eigenvalues of \eqref{Dirac-eigenvalue-problem} as
\[
\Sigma_{\ig,u}:=\big\{\lm\in\R:\, \ker\big(D_\ig-\lm u^{\frac2{m-2}}\big)\neq\{0\}\big\},
\]
and we call $\ker\big(D_\ig-\lm u^{\frac2{m-2}}\big)$ the eigenspace of $\lm\in\Sigma_{\ig,u}$.
The main results of this section are built upon the following observation. 
 
\begin{Prop}\label{prop:Sigma_u}
For each $u\in\cc_+^k$, the following holds:
	\begin{itemize}
		\item[$(1)$] The set $\Sigma_{\ig,u}$ is a closed subset of\, $\R\setminus\{0\}$ consisting of an unbounded discrete sequence of eigenvalues, in particular, $\Sigma_{\ig,u}$ is unbounded on both sides of\, $\R$.
		
		\item[$(2)$] The eigenspace of $\lm\in\Sigma_{\ig,u}$ is finite-dimensional and consists of eigenspinors in the class $\cc^{k+1}$.
		
		\item[$(3)$] If the new inner product
		\[
		(\psi,\va)^u_2=\real\int_M u^{\frac2{m-2}}(\psi,\va)_\ig\, d\vol_\ig
		\]
		is introduced on $L^2(M,\mbs(M))$, then the eigenspaces of all $\lm\in\Sigma_{\ig,u}$ form a complete orthonormal decomposition of $L^2(M,\mbs(M))$, i.e.,
		\[
		L^2(M,\mbs(M))=\ov{\bigoplus_{\lm\in\Sigma_{\ig,u}}\ker\big(D_\ig-\lm u^{\frac2{m-2}}\big)}.
		\] 
		where the close is taken over with respect to the $(\cdot,\cdot)_2^u$-norm.
\end{itemize}
\end{Prop}


\begin{Rem}\label{rem:Sigma_u}
There are several ways of dealing with the set $\Sigma_{\ig,u}$, and here we point out that the eigenvalue problem $D_\ig\psi=\lm u^{\frac2{m-2}}\psi$ can be conformally transformed into
\begin{\equ}\label{conformally-transformed-equ}
D_{\ig^u}\va=\lm\va
\end{\equ}
where $\ig^u=u^{\frac4{m-2}}\ig$ and $\va=u^{-\frac{m-1}{m-2}}\psi$. The new inner product $(\cdot,\cdot)^u_2$ is actually the standard $L^2$-inner product induced by the metric $\ig^u$. And hence $\Sigma_{\ig,u}$ is nothing but the spectrum of $D_{\ig^u}$. With all these in hand, we find Proposition \ref{prop:Sigma_u} is a direct consequence of the classical spectral theory of elliptic self-adjoint operators. And for later use, we shall write $\Sigma_{\ig,u}=\{\lm_i(u):\,i\in\Z\}$ with the ordering
\[
-\infty\leftarrow\cdots\leq \lm_{-1}(u)\leq \lm_0(u)\leq0<\lm_1(u)\leq\lm_2(u)\leq\cdots\rightarrow+\infty
\]
where the eigenvalues are repeated with respect to their multiplicities. We also remark that these eigenvalues are not necessarily symmetric about the origin, see \cite{Ginoux}.
\end{Rem}

\subsection{A quick review on the multiplicity of eigenvalues}

When it comes to perturbing the function $u$ in order to describe the behavior of the eigenvalue problem \eqref{Dirac-eigenvalue-problem}, it will be beneficial and simplifying to assume that the eigenvalues and eigenspinors vary smoothly with respect to the external factors. But things are often more complicated, and here we refer the readers to Kato's book \cite{Kato} for a technical report of general perturbed eigenvalue problems. In what follows, we give an elementary example of $2\times2$-matrices to illustrate that  perturbation of eigenvectors may have nasty behavior.
\begin{Ex}
Let $I$ be the $2\times2$ identity matrix. Then any vector is an eigenvector, and we may simply set $v_0=(\sqrt2/2,\sqrt2/2)$ as one possible candidate with unit length. And if we consider a small perturbation 
\[
I(\vr)=I+\begin{pmatrix}
	\vr&0\\
	0&0
\end{pmatrix},
\]
then the unit eigenvectors will  be $v_1=(1,0)$ and $v_2=(0,1)$ which are independent of $\vr$. Clearly, $|v_0-v_1|$ and $|v_0-v_2|$ will never go to zero as $\vr\to0$. 
\end{Ex}

From the above example, we can observe that a small perturbation can cause a repeated eigenvalue to split into distinct new eigenvalues. And additionally, an arbitrary choice of the unperturbed eigenvector may lead to discontinuity in relation to the perturbation parameter.

Regarding the conformally covariant operators, we are lucky that, in many cases, it has been shown that the eigenvalues of these operators are generically simple. The main example is the Laplace operator on smooth functions on a compact manifold, see for instance \cite{BU, BW, Teytel, Uh}, and it is generally believed that eigenvalues of these formally self-adjoint operators with positive leading symbol on $SO(m)$ or $Spin(m)$ irreducible bundles are generically simple.

Now, let us turn back to the equation \eqref{conformally-transformed-equ}, which is conformally equivalent to our target problem \eqref{Dirac-eigenvalue-problem}. We remark here that, when $M$ has dimension $m=3,4,5 \,(\text{mod } 8)$, the spinor bundle has a quaternionic structure which commutes with Clifford multiplication~\cite[Section 1.7]{Friedrich00} or \cite[Page 33, Table III]{LM}. Hence, in these cases the eigenspaces of the Dirac operator are quaternionic vector spaces. So in the sequel when we are talking about the dimension of an eigenspace, we mean the quaternionic dimension $\dim_{\bH}$ for $m=3,4,5 \,(\text{mod } 8)$ and the complex dimension $\dim_\C$ for other cases.

For a closed spin $m$-manifold $M$ and $k\in\N\cup\{\infty\}$, denote by $\cm^k$ the space of Riemannian metrics on $M$ equipped with the $C^k$-topology. In what follows, for a given $\ig\in\cm^k$, let us consider the conformal change $\ig^u=u^{\frac4{m-2}}\ig$ with $u\in \cc_+^k$ and the possible splitting of the eigenvalues of $D_{\ig^u}$ when $u$ is a perturbation from some $u_0$. We introduce the following
\begin{Def}
{\it An eigenvalue $\lm$ of $D_{\ig^{u_0}}$ (i.e., $\lm\in\Sigma_{\ig,u_0}$) is said to have the conformal splitting property in the direction of $v\in C^\infty(M)$ if $\dim\ker(D_{\ig^{u_0}}-\lm)>1$ and there exists at least one pair of eigenvalues $\lm_{0,1}(\vr)$ and $\lm_{0,2}(\vr)$ of $D_{\ig^{u_0+\vr v}}$, which are continuous functions defined for small $\vr\geq0$ satisfying
\begin{itemize}
	\item[$(1)$] $\lm_{0,1}(0)=\lm_{0,2}(0)=\lm$;
	
	\item[$(2)$] $\lm_{0,1}(\vr)\neq\lm_{0,2}(\vr)$ for $\vr>0$.
\end{itemize}
The eigenvalues $\lm_{0,1}(\vr)$ and $\lm_{0,2}(\vr)$ are said to depart from the unperturbed eigenvalue $\lm$ by splitting at $\vr=0$.
}
\end{Def}
\begin{Rem}\label{Remark-Def}
Here we emphasize that, since the mapping $\vr\mapsto \ig^{u_0+\vr v}$ provides a real analytic parameterization of metrics in $\cm^k$, the perturbation theory developed in \cite[Chapter VII]{Kato} can be employed. In particular, each $\lm\in\Sigma_{\ig,u_0}$ provides a batch of real analytic functions $\lm_{0,1}(\vr),\dots,\lm_{0,l}(\vr)\in \Sigma_{\ig,u_0+\vr v}$ such that $\lm_{0,j}(0)=\lm$, $j=1,\dots,l$, and the sum of all multiplicities of $\lm_{0,j}(\vr)$'s equals to the multiplicity of $\lm$. And therefore, if $\lm\in\Sigma_{\ig,u_0}$ has the conformal splitting property in a direction $v$ then the multiplicity of each eigenvalue, which depart from $\lm$, will be strictly less than that of $\lm$. 
Meanwhile, those eigenvalue that has simple multiplicity will stay simple as $\vr$ varies.
\end{Rem}

In this setting, we have the following lemma which is due to M. Dahl \cite[Lemma 3.1]{Dahl}.

\begin{Lem}\label{Dahl-lemma}
On a closed spin $3$-manifold, given $\ig\in\cm^k$ and let $\lm$ be a non-zero eigenvalue of $D_\ig$ with $\dim_{\bH}\ker(D_\ig-\lm)>1$. Then there exists $v\in C^\infty(M)$ such that $\lm$ has the conformal splitting property in the direction of $v$.
\end{Lem}

\begin{Rem}
\begin{itemize}
	\item[(1)] Let us point out here that, with an additional rigid assumption on the operator $D_\ig$, a similar result holds for higher dimensions (see \cite{Canzani} for detailed statements). And it is unclear whether or not such a rigid assumption can be dropped in general cases. The reason that one does not need further assumptions in dimension $3$ is because the quaternionic rank of the spinor bundle is simple, that is $\rank_{\bH}\mbs(M)=1$.
	
	\item[(2)] Due to the conformal covariance of the Dirac operator (see \eqref{conformal-Dirac}), $\dim_{\bH}\ker D_\ig$ is invariant under the conformal change of the background metric $\ig$. This is the reason why we only focus on non-zero eigenvalues and the associated eigenspaces.
\end{itemize}
\end{Rem}

By recalling the conformal equivalence of \eqref{Dirac-eigenvalue-problem} and \eqref{conformally-transformed-equ}, for given $k,n\in\N$, $\ig\in\cm^k$, let us consider the set
\begin{\equ}\label{cu-g-n-k}
\cu_{\ig,n}^k:=\big\{ u\in \cc_+^k:\, \lm \text{ is simple for all } \lm\in\Sigma_{\ig,u}\cap([-n,n]\setminus\{0\})  \big\}.
\end{\equ}
Then, in a view of Remark \ref{Remark-Def}, we have the following characterization for $\cu_{\ig,n}^k$.
%
%

\begin{Lem}\label{lem:open-dense}
On a closed spin $m$-manifold with $m\geq3$ and, if $m>3$, suppose additionally that every non-zero eigenvalue of $D_{\tilde\ig}$ has either simple multiplicity or the conformal splitting property in certain directions, for any $\tilde\ig\in [\ig]:=\{e^{2f}\ig:\, f\in\cc_+^k\}$. Then $\cu_{\ig,n}^k$ is an open dense subset of $\cc_+^k$.
\end{Lem}
\begin{proof}
We only need to show that $\cu_{\ig,n}^k$ is dense in $\cc_+^k$, since the openness follows from the fact that the eigenvalues of $D_\ig$ depend continuously on $\ig$ (see \cite[Proposition 7.1]{Bar1996}). To this end, let us fix $u_*\in \cc_+^k\setminus\cu_{\ig,n}^k$ and let $U$ be an open neighborhood of $u_*$. Since there exists $\lm_*\in\Sigma_{\ig,u_*}\cap([-n,n]\setminus\{0\})$ such that 
\[
\dim_{\bH}\ker(D_{\ig^{u_*}}-\lm_*)>1,
\]
we can use our assumption to obtain $v_1\in C^\infty(M)$ and a conformal deformation $\ig^{u_*+\vr_1 v_1}$ which decreases the multiplicity of $\lm_*$, for $\vr_1>0$ suitably small. Clearly, those eigenvalues in $[-n,n]\setminus\{0\}$ that were simple would remain being simple for such $\vr_1$ due to the analytic perturbation theory (see \cite{Kato}). Let $\vr_1$ be small such that $u_*+\vr_1 v_1\in U$. If $u_*+\vr_1 v_1\in \cu_{\ig,n}^k$, then we are done. And if not, we can repeat the above procedure finitely many times to get a function $u_*+\vr_1 v_1+\cdots+\vr_l v_l\in U\cap \cu_{\ig,n}^k$. Therefore, due to the arbitrariness of $u_*$ and $U$, we have $\cu_{\ig,n}^k$ is dense in $\cc_+^k$ as desired.
\end{proof}


\subsection{The perturbation}

For a given $u_0\in\cc_+^k$, we will now stick to the case that $\lm_0\in \Sigma_{\ig,u_0}\setminus\{0\}$ is a simple eigenvalue and $\psi_0$ is a normalized eigenspinor, that is we have 
\begin{\equ}\label{p0}
\ker\big(D_\ig-\lm_0u_0^{\frac2{m-2}}\big)=\spann\{\psi_0\} \quad \text{and} \quad \int_Mu_0^{\frac2{m-2}}|\psi_0|_\ig^2d\vol_{\ig}=1.
\end{\equ}
Actually, there is no loss of generality if we assume that $u_0\in\cu_{\ig,n_0}^k$ for some fixed $n_0\geq1$ and $|\lm_0|\leq n_0$, and we shall work under the assumptions  of Lemma \ref{lem:open-dense}.
It is clear that $L^2(M,\mbs(M))$ possesses the $(\cdot,\cdot)_2^{u_0}$-orthogonal decomposition $L^2(M,\mbs(M))=\spann\{\psi_0\}\op\spann\{\psi_0\}^\bot$.

To proceed, let us introduce the spinor subspaces $\tilde L^2(M,\mbs(M))=\spann_\R\{\psi_0\}\op\spann\{\psi_0\}^\bot$ and $\tilde H^1(M,\mbs(M))=H^1(M,\mbs(M))\cap\tilde L^2(M,\mbs(M))$. That is, we consider the subspaces in which the $\psi_0$-direction is always real. Furthermore,  we will consider the mapping  $\Phi:\cc_+^k\times\R\times \tilde H^1(M,\mbs(M))\to\R\times \tilde L^2(M,\mbs(M))$ with
\begin{\equ}\label{Phi}
\Phi(u,\lm,\psi)=\Big( \int_Mu^{\frac2{m-2}}|\psi|_\ig^2d\vol_{\ig}-1,\, u^{-\frac2{m-2}}D_\ig\psi-\lm \psi \Big).
\end{\equ}
Then, by taking derivatives with respect to $\lm$ and $\psi$, we find
\[
\nabla_{(\lm,\psi)}\Phi(u,\lm,\psi)[\mu,\va]=\Big( 2\real\int_M u^{\frac2{m-2}}(\psi,\va)_\ig d\vol_{\ig},\, u^{-\frac2{m-2}}D_\ig\va-\lm \va-\mu \psi\Big).
\]
for $\mu\in\R$ and $\va\in H^1(M,\mbs(M))$. And in what follows, we are going to use the Implicit function theorem to characterize the relation between the external factor $u$ and the perturbed couple $(\lm,\psi)$ in the eigenvalue problem \eqref{Dirac-eigenvalue-problem}.

Since we have assumed the validity of \eqref{p0} for $(u_0,\lm_0,\psi_0)$, we can write $\va=\tau\psi_0+\va^\bot$ with $\tau\in\R$ and $\va^\bot\in \spann\{\psi_0\}^\bot$ for any $\va\in \tilde H^1(M,\mbs(M))\subset\tilde L^2(M,\mbs(M))$ with respect to the inner product $(\cdot, \cdot)^{u_0}_2$. Then we have
\[
\nabla_{(\lm,\psi)}\Phi(u_0,\lm_0,\psi_0)[\mu,\va]=\big( 2\tau,\, u_0^{-\frac2{m-2}}D_\ig\va^\bot-\lm_0 \va^\bot-\mu \psi_0\big).
\]
At this moment, by collecting $\{\va_i\}_{i\in\Z}$ the $(\cdot,\cdot)_2^{u_0}$-normalized eigenspinors associated to $\Sigma_{\ig,u_0}$ (see Proposition \ref{prop:Sigma_u} and Remark \ref{rem:Sigma_u}), we can represent for an element $\phi\in \tilde L^2(M,\mbs(M))$ with the Fourier series $\phi=\sum \al_i\va_i$, where $\al_i\in\C$. Clearly, we have $\lm_{i_0}(u_0)=\lm_0$ and $\va_{i_0}=\psi_0$ for some $i_0$,  and hence $\al_{i_0}=(\phi,\psi_0)_2^{u_0}\in\R$. Then, by looking at the equation
\[
\nabla_{(\lm,\psi)}\Phi(u_0,\lm_0,\psi_0)[\mu,\va]=(\ka,\phi)
\]
with $\ka\in\R$ and $\phi\in \tilde L^2(M,\mbs(M))$ being arbitrarily given, we can solve it uniquely with 
\[
\mu=-\al_{i_0} \quad \text{and} \quad
\va=\frac\ka2\psi_0+\sum_{i\neq i_0}\frac{\al_i}{\lm_i(u_0)-\lm_0}\va_i.
\] 
Therefore $\nabla_{(\lm,\psi)}\Phi(u_0,\lm_0,\psi_0):\R\times \tilde H^1(M,\mbs(M))\to \R\times \tilde L^2(M,\mbs(M))$ is invertible. And now, a direct application of the Implicit function theorem gives us the following perturbation result.

\begin{Lem}\label{lem:Ck-perturbed-eigenvalue}
Given $k,n_0\in\N$, let $u_0\in\cu_{\ig,n_0}^k$ and $(\lm_0,\psi_0)$ satisfies \eqref{p0} with $\lm_0\in[-n_0,n_0]\setminus\{0\}$. If $u\in C^k(M\times[0,T])$ with $u(\cdot,0)=u_0$, then (by narrowing the time interval if necessary) there exist a $C^k$ mapping $[0,T]\to(\lm(t),\psi(t))$ such that $\lm(0)=\lm_0$, $\psi(0)=\psi_0$ and 
\begin{\equ}\label{Ck-perturbed-identity}
\Phi(u(\cdot,t),\lm(t),\psi(t))\equiv0 \quad \text{for }t\in[0,T].
\end{\equ}
Furthermore, 
apart from the function $u$ and its time-derivatives,  the $j$-th derivatives of $\lm$ and $\psi$ only depend on the first $j-1$ derivatives of $\lm$ and $\psi$, where $1\leq j\leq k$. In particular, the first order derivatives of $\lm$ and $\psi$ can be formulated as
\[
\lm'(t)=-\frac{2\lm(t)}{m-2}\int_M u(\cdot,t)^{\frac{4-m}{m-2}}\pa_t u(\cdot,t)|\psi(t)|_\ig^2\, d\vol_{\ig} 
\]
and
\[
\psi'(t)=\frac{\lm'(t)}{2\lm(t)}\psi(t)-\frac{2\lm(t)}{m-2}\big( u(t)^{-\frac2{m-2}}D_\ig-\lm(t) \big)^{-1}\circ(I-P_{\lm(t)})\big( u(t)^{-1}\pa_tu(\cdot,t)\psi(t) \big)
\]
for all $t\in[0,T]$, where $P_{\lm(t)}$ is the $L^2$-eigenprojection associated to $\lm(t)$. 
\end{Lem}

\section{A nonlocal linear parabolic equation}\label{sec:nonlocal linear equ}


In this section, we  discuss the existence of solutions to a linear equation of the form
\begin{\equ}\label{a linear problem}
	\left\{
	\aligned
	&\pa_t u - \ca(x,t)\De_\ig u+\cl[u]=f(x,t) \\
	& u(x,0)=u_0(x)
	\endaligned
	\right.	
\end{\equ}
where $\ca:M\times[0,+\infty)\to\R$ are smooth and uniformly bounded with all their derivatives, and $\cl$ is a continuous linear operator between some suitable function spaces containing lower order differential operations (a specific characterization of $\cl$ will be given later). Moreover, let us assume that the function $\ca$ is uniformly positive, i.e., there exists a positive constant $\de>0$ such that $\ca(x,t)\geq\de$ for all $x$ and $t$. Hence Eq. \eqref{a linear problem} is of parabolic type. This developments in this section serve as the PDE foundations for the flow we investigate.

In order to give a precise description of the linear operator $\cl$ and to obtain  solutions of Eq. \eqref{a linear problem}, let us introduce some basic function spaces that are useful in our argument. Let $u,v:M\times[0,+\infty)\to\R$ be smooth functions, we set
\[
\inp{u}{v}_{LH_a^k}=\int_0^\infty e^{-2at}\inp{u(\cdot,t)}{v(\cdot,t)}_{H^k(M)}dt \quad \text{for } k=0,1,2,\dots
\]
and
\[
\inp{u}{v}_{HH_a}=\inp{u}{v}_{LH_a^1}+\inp{\pa_t u}{\pa_t v}_{LH_a^0},
\]
where $a>0$ is understood to be a positive parameter. We denote $LH_a^k$ and $HH_a$ to be the Hilbert spaces formed by completion of $C^\infty(M\times[0,+\infty))$ in the corresponding norms. We also consider $C_c^\infty(M\times(0,+\infty))$ as the space of smooth functions which vanish for very large and very small times, and denote $HH_{a,0}$ the completion of $C_c^\infty(M\times(0,+\infty))$ in $HH_a$. Furthermore, when we come to consider the higher regularity of solutions of Eq. \eqref{a linear problem}, we shall make use of the following Hilbert space
\[
P_a^l:=\big\{ u:M\times[0,+\infty)\to\R:\, \|\pa_t^ju\|_{LH_a^{2(l-j)}}<+\infty,\ j=0,\dots,l \big\}
\]
and equip it with the inner product
\[
\inp{u}{v}_{P_a^l}=\sum_{j=0}^l\inp{\pa_t^ju}{\pa_t^jv}_{LH_a^{2(l-j)}}=\sum_{2j+k\leq 2l}\int_0^\infty e^{-2at}\inp{\pa_t^j u}{\pa_t^jv}_{H^k(M)}dt.
\]
We note that $P_a^0=LH_a^0$ and the embedding $P_a^l\hookrightarrow LH_a^{2l}$ holds for $l\geq1$. And we also have the embedding $P_a^l(M)\hookrightarrow HH_a$ for $l\geq1$.

Now we turn to characterize the linear operator $\cl$ in Eq. \eqref{a linear problem}, and we shall make the following three standing assumptions under consideration. For ease of notation, the space of bounded linear operators from a Banach space $X$ to another Banach space $Y$ is denoted by $\msl(X,Y)$.

\begin{itemize}
	\item[(A1)] $\cl\in \msl(LH_a^1, LH_a^0)$, moreover, there exists a time-independent constant $C>0$ such that $|\cl[u](\cdot,t)|_2\leq C\|u(\cdot,t)\|_{H^1(M)}$ for all $u\in LH_a^1$. 
	
	\item[(A2)] $\cl[\al u](\cdot,t)=\al(t)\cl[u](\cdot,t)$ for any smooth function $\al:[0,+\infty)\to\R$
	
	\item[(A3)] $\cl\in\msl(P_a^l,P_a^{l})$ for $l\in\N$
\end{itemize}

Using the specific form of the operator $-\ca(x,t)\De_\ig+\cl$ and integration by parts, we introduce the following bilinear form, which we denote by $\msa_t$, defined on $H^1(M)\times H^1(M)$ (this product space is understood as a $t$-fiber of the total space $LH_a^1\times LH_a^1$):
\[
\aligned
\msa_t(u,v)&=\int_M-\ca v\De_\ig u+ v\cl[u]d\vol_\ig \\
&=\int_M\ca \nabla_\ig u\cdot\nabla_\ig v + v\nabla_\ig\ca \cdot\nabla_\ig u+v\cl[u]d\vol_\ig.
\endaligned
\]
Then a weak solution of Eq. \eqref{a linear problem} is defined by a function $u\in HH_a$ such that $u(x,0)=u_0(x)$ and
\begin{\equ}\label{def-weak-solu-linear}
	\inp{\pa_t u}{\phi}_{LH_a^0}+\int_0^\infty e^{-2at}\msa_t(u,\phi)dt=\inp{f}{\phi}_{LH_a^0}
\end{\equ}
for all $\phi\in C_c^\infty(M\times(0,+\infty))$. 

Our first existence result is stated as follows.

\begin{Prop}\label{prop:existence in HH-a}
	Let $\ca$ and $\cl$ be as above. There exists $a_0>0$ depending only on $\ca$, $\cl$ and the background manifold $(M,\ig)$ such that, if $a>a_0$,  then for every $f\in LH_a^0$ the equation
	\begin{\equ}\label{equ-initial-0}
		\pa_t u -\ca(x,t)\De_\ig u + \cl[u]=f, \quad u(\cdot,0)=0
	\end{\equ}
	has a unique weak solution in $HH_{a,0}$.
\end{Prop}

Before establishing the proof of Proposition \ref{prop:existence in HH-a}, let us recall a variant of the Lax-Milgram Theorem, which relaxes the continuity assumption on the bilinear form. 

\begin{Lem}\label{lem:Lax}
	Let $(H,\|\cdot\|_H)$ be a Hilbert space and $(V,\|\cdot\|_V)$ be an inner-product space continuously embedded in $H$. Let $\Phi:H\times V\to\R$ be a bilinear form with the properties that
	\begin{itemize}
		\item[$(1)$] the mapping $h\mapsto\Phi(h,v)$ is continuous on $H$ for each fixed $v\in V$,
		
		\item[$(2)$] $\Phi$ is coercive on $V$, i.e., $\Phi(v,v)\geq \bt\|v\|_V^2$ for some $\bt>0$. 
	\end{itemize}
	If $F:V\to\R$ is a continuous linear functional, then there exists $h_F\in H$ such that $F(v)=\Phi(h_F,v)$ for all $v\in V$. If, in addition $V$ is dense in $H$, then $h_F\in H$ is uniquely determined.
\end{Lem}

\begin{Rem}
	In Lemma \ref{lem:Lax}, we do not assume that $(V,\|\cdot\|_V)$ is complete. And we also do not need the bi-continuity of the bilinear form $\Phi$, which is a critical hypothesis of the Lax-Milgram Theorem. We refer the readers to \cite[Chapter 10, Theorem 16]{Friedman} and \cite[Lemma 3.2]{Sharples} for a detailed proof of Lemma \ref{lem:Lax}.
\end{Rem}

\begin{proof}[Proof of Proposition \ref{prop:existence in HH-a}]
	First of all, let us remark that the assumptions on $\ca$ and $\cl$ imply that a G\r{a}rding type inequality holds for the operator $-\ca(x,t)\De_\ig+\cl$, that is
	\begin{\equ}\label{Garding-inequ}
		\msa_t(\phi,\phi)\geq\frac\de2\|\phi\|_{H^1(M)}^2- \ka |\phi|_2^2 
	\end{\equ}
	where $\ka>0$ is a constant depending continuously on the $C^1$-norm of $\ca$ and on the curvature tensor of $(M,\ig)$ and its covariant derivatives. And, by virtue of the assumption $(A2)$, we find that $u$ is a solution to \eqref{equ-initial-0} if and only if $w(x,t)=e^{-\ga t}u(x,t)$ solves the equation
	\begin{\equ}\label{equivalent linear equ}
		\pa_t w-\ca(x,t)\De_\ig w +\ga w+\cl[w]=f_\ga(x,t):=e^{-\ga t}f(x,t), \quad w(\cdot,0)=0.
	\end{\equ}
	By choosing $\ga=\ka$, we see that
	\begin{\equ}\label{coercive}
		(\msa_t+\ga)(\phi,\phi)=\msa_t(\phi,\phi)+\ga|\phi|_2\geq\frac\de2\|\phi\|_{H^1(M)}^2.
	\end{\equ}
	In the sequel, we will focus on the equivalent problem \eqref{equivalent linear equ}. 
	
	To proceed, let us consider a bilinear form on $HH_{a,0}\times C_c^\infty(M\times(0,+\infty))$:
	\[
	\msp(w,\phi)=\inp{\pa_tw}{\pa_t\phi}_{LH_a^0}+\int_0^\infty e^{-2at}(\msa_t+\ga)(w,\pa_t\phi)dt,
	\]
	and a linear functional on $C_c^\infty(M\times(0,+\infty))$:
	\[
	F(\phi)=\inp{f_\ga}{\pa_t\phi}_{LH_a^0}.
	\]
	We mention here that the parameter $a>0$ will be given later. Clearly, for a fixed $\phi\in C_c^\infty(M\times(0,+\infty))$, the bilinear form $\msp$ is continuous in $w$ with respect to the $HH_a$-norm. And it is evident that $F$ is continuous with with respect to the $HH_a$-norm, given that $f_\ga\in LH_a^0$. 
	
	To see the coerciveness of $\msp$ on $C_c^\infty(M\times(0,+\infty))$, for an arbitrary $\phi\in C_c^\infty(M\times(0,+\infty))$, let us look at
	\[
	\msp(\phi,\phi)=\|\pa_t\phi\|_{LH_a^0}^2+\int_0^\infty e^{-2at}(\msa_t+\ga)(\phi,\pa_t\phi)dt.
	\]
	Let $I$ denote the second term on the right hand side, i.e.,
	\[
	\aligned
	I&=\int_0^\infty e^{-2at}(\msa_t+\ga)(\phi,\pa_t\phi)dt \\
	&=\int_0^\infty e^{-2at}\int_M \big(\ca \nabla_\ig \phi\cdot\nabla_\ig\pa_t\phi + \pa_t\phi\nabla_\ig\ca \cdot\nabla_\ig\phi +\ga\phi\pa_t\phi +\pa_t\phi\cl[\phi]\big)d\vol_\ig dt.
	\endaligned
	\]
	By setting
	\[
	I_1=\int_0^\infty e^{-2at}\int_M \big(\ca \nabla_\ig \phi\cdot\nabla_\ig\pa_t\phi  +\ga\phi\pa_t\phi \big)d\vol_\ig dt
	\]
	and
	\[
	I_2=\int_0^\infty e^{-2at}\int_M \big( \pa_t\phi\nabla_\ig\ca \cdot\nabla_\ig \phi +\pa_t\phi\cl[\phi]\big)d\vol_\ig dt,
	\]
	we soon have $I=I_1+I_2$ and 
	\[
	\aligned
	I_2&\geq -(|\nabla_\ig\ca|_\infty+C)\int_0^\infty e^{-2at}|\pa_t\phi(\cdot,t)|_{2}\|\phi(\cdot,t)\|_{H^1(M)}dt \\
	&\geq - \frac12\|\pa_t\phi\|_{LH_a^0}^2-\frac{(|\nabla_\ig\ca|_\infty+C)^2}2\|\phi\|_{LH_a^1}^2
	\endaligned
	\]
	where $C>0$ is the constant from the assumption $(A1)$ for the linear operator $\cl$ and the last inequality comes from the Young's inequality. On the other hand, it follows from partial integration in time that
	\[
	\aligned
	I_1&=\int_0^\infty e^{-2at}\int_M\Big( a\ca|\nabla_\ig\phi|^2+a\ga|\phi|^2 -\frac12\pa_t\ca|\nabla_\ig\phi|^2\Big)d\vol_\ig dt \\
	&\geq \Big( a\min\{\de,\ga\}-\frac12|\pa_t\ca|_\infty \Big)\|\phi\|_{LH_a^1}.
	\endaligned
	\]
	Thus, by combining the above two estimates, we obtain
	\[
	\msp(\phi,\phi)\geq \frac12\|\pa_t\phi\|_{LH_a^0}^2+\Big( a\min\{\de,\ga\}-\frac12|\pa_t\ca|_\infty -\frac{(|\nabla_\ig\ca|_\infty+C)^2}2\Big)\|\phi\|_{LH_a^1}.
	\]
	And, when $a$ is chosen so large that \[
	a>\frac{|\pa_t\ca|_\infty+(|\nabla_\ig\ca|_\infty+C)^2}{2\min\{\de,\ga\}},
	\]
	we can guarantee that $\msp$ is coercive on $C_c^\infty(M\times(0,+\infty))$ with respect to the $HH_a$-norm. Therefore we may conclude from Lemma \ref{lem:Lax} that there is an unique $w_F\in HH_{a,0}$ for which $\msp(w_F,\phi)=F(\phi)$ for all $\phi\in C_c^\infty(M\times(0,+\infty))$.
	
	We are almost at the stage to deliver the unique weak solution of \eqref{equ-initial-0}. The problem that remains is we have only tested \eqref{equivalent linear equ} against functions with zero average over time, that is, the bilinear form $\msp$ is defined by simply testing \eqref{equivalent linear equ} with the function $\pa_t\phi$. This would normally mean that the $w_F$ obtained above differs from a solution of \eqref{equivalent linear equ} by a time constant function. However, one does not need to worry about this problem. In fact, for any $\va\in C_c^\infty(M\times(0,+\infty))$, we can consider a modified function
	\[
	\va_\tau(x,t)=\va(x,t)-\va(x,t-\tau)
	\]
	where $\tau>0$ is large enough such that the support of the second term does not intersect with that of the first term. Then $\va_\tau$ can be represented as $\pa_t\phi_\tau$ for some $\phi_\tau\in C_c^\infty(M\times(0,+\infty))$, and hence
	\[
	\inp{\pa_tw_F}{\va_\tau}_{LH_a^0}+\int_0^\infty e^{-2at}(\msa_t+\ga)(w_F,\va_\tau)dt=\inp{f_\ga}{\va_\tau}_{LH_a^0}
	\]
	holds for sufficiently large $\tau$. By taking to the limit $\tau\to+\infty$, it follows that
	\[
	\inp{\pa_tw_F}{\va}_{LH_a^0}+\int_0^\infty e^{-2at}(\msa_t+\ga)(w_F,\va)dt=\inp{f_\ga}{\va}_{LH_a^0}.
	\]
	Hence $w_F$ is indeed the unique weak solution of \eqref{equivalent linear equ}, and this is to say, there exists an unique solution $u\in HH_{a+\ga,0}$ of Eq. \eqref{equ-initial-0} provided that $f\in LH_{a+\ga}^0$. 
\end{proof}

As an immediate consequence of Proposition \ref{prop:existence in HH-a}, we have an existence and uniqueness result for more general initial data.

\begin{Cor}\label{cor:existence HH-a}
	Under the hypotheses of Proposition \ref{prop:existence in HH-a}. For every $f\in LH_a^0$ and $u_0\in H^2(M)$, the equation
	\begin{\equ}\label{equ-initial-general}
		\pa_tu-\ca(x,t)\De_\ig u + \cl[u]=f, \quad u(\cdot,0)=u_0
	\end{\equ} 
	has a unique weak solution in $HH_a$, provided $a>a_0$ as in Proposition \ref{prop:existence in HH-a}.
\end{Cor}
\begin{proof}
	Let us set $\tilde u_0(x,t)=e^{-t}u_0(x)$ and $v=u-\tilde u_0$. Then we shall consider the equation
	\begin{\equ}\label{modified general equ}
		\pa_t v-\ca(x,t)\De_\ig v+\cl[v]=f+\tilde u_0+\ca(x,t)\De_\ig \tilde u_0-\cl[\tilde u_0], \quad v(\cdot,0)=0.
	\end{\equ}
	Since $u_0\in H^2(M)$, we soon have  $\tilde u_0+\ca \De_\ig\tilde u_0\in LH_a^0$. Note also that $\tilde u_0\in P_a^1$, by virtue of the assumption $(A3)$ for $\cl$, we have $\cl[\tilde u_0]\in P_a^1(M)\hookrightarrow P_a^0(M)=LH_a^0$. Hence, Proposition \ref{prop:existence in HH-a} implies that for all $a>0$ large enough, there is a unique $v\in HH_{a,0}$ solving Eq. \eqref{modified general equ}, which is equivalent to say that there is an unique $u\in HH_a$ solving \eqref{equ-initial-general}.
\end{proof}

Let us mention here that the above existence result is less than optimal as the natural class of the initial data is $H^1(M)$. However, these limitations will be addressed and improved by the estimates proven in our upcoming work. Specifically, our next focus is on the regularity of the weak solution that has been obtained above. In essence, we aim to establish two main results: firstly, the weak solution exhibits a level of smoothness determined by the homogeneous term $f$ and the initial value $u_0$; and secondly, the relationship between the solution $u$ and the homogeneous term can be understood as an isomorphism within suitably defined Banach spaces.

\begin{Lem}\label{lem:regularity1}
	For $a\geq\ka+\frac12$, if $u\in HH_a$ is a weak solution of \eqref{equ-initial-general} with $f\in LH_a^0$ and $u_0\in L^2(M)$, then the following estimate holds
	\[
	\|u\|_{LH_a^1}^2\leq \frac1\de\big( |u_0|_2^2+\|f\|_{LH_a^0}^2 \big).
	\]
	where $\ka>0$ is given in \eqref{Garding-inequ} and $\de>0$ is the lower bound of the function $\ca$.
\end{Lem}
\begin{proof}
	Since $u\in HH_a$ is a weak solution of \eqref{equ-initial-general}, by using $u$ itself as a test function, we find that
	\begin{\equ}\label{test1}
		\inp{\pa_t u}{u}_{LH_a^0}+\int_0^\infty e^{-2at}\msa_t(u,u)dt=\inp{f}{u}_{LH_a^0}.
	\end{\equ}
	Although one should use a compactly supported smooth function to test \eqref{equ-initial-general}, 
	the above identity is still possible since the density of $C_c^\infty(M\times(0,+\infty))$ in $HH_{a,0}$  suggests us to use $\chi_\vr\cdot u\in HH_{a,0}$ as a test function, where $\chi_\vr$ is a smooth cut-off multiplier given by $\chi_\vr(t)=\chi(t/\vr)$ with
	\[
	\chi(t)=\begin{cases}
		0, & t\leq 1,\\
		1, & t\geq2.
	\end{cases}
	\] 
	Then, by taking to the limit $\vr\searrow0$, one easily obtain \eqref{test1}.

	Notice that, for any $v\in HH_a$ with initial value $v(\cdot,0)\in L^2(M)$, we have the identity
	\begin{\equ}\label{identity1}
		\aligned
		\inp{\pa_t v}{v}_{LH_a^0}&=\int_0^\infty e^{-2at}\int_M v\pa_tv d\vol_\ig dt\\
		&=a\int_0^\infty e^{-2at}\int_M v^2 d\vol_\ig dt -\frac12 \int_M v(\cdot,0)^2d\vol_\ig \\
		&=a\|v\|_{LH_a^0}^2-\frac12|v(\cdot,0)|_2^2.
		\endaligned
	\end{\equ}
	Hence, by \eqref{test1}, \eqref{Garding-inequ} and Young's inequality, we soon deduce that
	\[
	\frac\de2\|u\|_{LH_a^1}^2+\Big( a-\ka-\frac12 \Big)\|u\|_{LH_a^0}^2\leq \frac12|u_0|_2^2+\frac12\|f\|_{LH_a^0}^2.
	\]
	Therefore, by choosing $a\geq \ka+\frac12$, we obtain the conclusion.
\end{proof}

The above estimate can be further improved to get higher regularity. In fact, we have the following $P_a^l$-version of the usual a priori energy estimate for parabolic
equations, whose proof is postponed to the Appendix \ref{A. linear regualrity}
\begin{Lem}\label{lem:regularity3}
	For $l\in\N\cup\{0\}$, if $u\in P_a^l$ is a weak solution of \eqref{equ-initial-general} with $f\in P_a^l$ and $u_0\in H^{2l+1}(M)$, then $u\in P_a^{l+1}$ and
	\[
	\|u\|_{P_a^{l+1}}^2\leq C\big( \|u_0\|_{H^{2l+1}(M)}^2+\|f\|_{P_a^l}^2 \big)
	\]
	provided $a>0$ is large enough in the sense of Proposition \ref{prop:existence in HH-a} and Lemma \ref{lem:regularity1}, where $C>0$ depends only on $l,\ca,\cl,\nabla_\ig\ca,\pa_t\ca,\dots,\pa_t^l\ca$ and the manifold $M$.
\end{Lem}

The estimates we have just derived can now be utilized to yield the following global existence result, which is an analogue of \cite[Theorem 2.3.5]{Polden}. 

\begin{Thm}\label{thm-mff-isomorphism}
	If $a>0$ is large enough, then for every $l\in\N\cup\{0\}$ the linear map
	\begin{\equ}\label{the map mff}
		\mff[u]=\big(u_0,\pa_tu-\ca\De_\ig u +\cl[u]\big)
	\end{\equ}
	where $u_0=u(\cdot,0)$, is an isomorphism of $P_a^{l+1}$ onto $H^{2l+1}(M)\times P_a^l$.
\end{Thm}
\begin{proof}
	Since $\mff$ is linear, the continuity follows immediately from the boundedness of $\mff$. The bound for the second component is much clearer since we have
	\[
	\aligned
	\|\pa_tu-\ca\De_\ig u +\cl[u]\|_{P_a^l}&\leq\|\pa_tu\|_{P_a^l}+\|\ca\De_\ig u\|_{P_a^l}+\|\cl[u]\|_{P_a^l} \\
	&\leq C\|u\|_{P_a^{l+1}} + \|\cl[u]\|_{P_a^{l+1}} \\
	&\leq C\|u\|_{P_a^{l+1}}
	\endaligned
	\]
	where we have used the assumption (A3) for $\cl$ in the last inequality. Meanwhile, the continuity of the first component of $\mff$ follows directly from the argument in \cite{Polden}.
	
	We next need to show that $\mff$ has an inverse. This is suffices to show that the equation $\mff[u]=(u_0,f)$, for arbitrarily given $u_0$ and $f$, is uniquely solvable in the appropriate class. However, this
	was exactly the conclusion of the existence results Proposition \ref{prop:existence in HH-a}, Corollary \ref{cor:existence HH-a} and the
	regularity theory culminating in Lemma \ref{lem:regularity3}. 
	
	The only remaining loophole is that, when $l=0$, the initial data in Corollary \ref{cor:existence HH-a} was assumed to be in $H^2(M)$. In order to close this, for an arbitrary $f\in P_a^0$, let $u_0\in H^1(M)$ and let $\{u_0^{(j)}\}_{j=1}^\infty$ be a sequence of smooth functions on $M$ converging to $u_0$ with respect to the $H^1$-norm. For each $j$, Corollary \ref{cor:existence HH-a} returns a function $u^{(j)}$ in $P_a^1$ which solves the equation $\mff[u^{(j)}]=(u_0^{(j)},f)$. Then, by virtue of Lemma \ref{lem:regularity3} with $l=0$, we have $\{u^{(j)}\}$ converge in $P_a^1$ to a limit function $u$. And, by the continuity of $\mff$, we have $\mff[u]=(u_0,f)$, i.e. the inverse of $\mff$ exists.
	
	Note that the estimate in Lemma \ref{lem:regularity3} also implies the inverse of $\mff$ is continuous, we therefore conclude that $\mff:P_a^{l+1}\to H^{2l+1}(M)\times P_a^l$ is an isomorphism.
\end{proof}

%
%
%
%
%

\section{An initial-value problem for the Einstein-Dirac equation}\label{sec:ED flow}

\subsection{Set up of the problem}\label{subsec:set-up-of-the-problem}

Recall that we denote $L_\ig=-c_m\De_\ig+\scal_\ig$ the conformal Laplace operator on $(M,\ig)$ with $c_m=\frac{4(m-1)}{m-2}$. For a conformal metric $\ig^u=u^{\frac4{m-2}}\ig$, we note that the scalar curvature of $\ig^u$ is related to the scalar curvature of $\ig$ by
\[
\scal_{\ig^u}=u^{-\frac{m+2}{m-2}}L_\ig u.
\]
Hence for $\ig^{uv}=v^{\frac4{m-2}}\ig^u=(uv)^{\frac4{m-2}}\ig$, by letting $w=uv$, we obtain
\[
\scal_{\ig^{uv}}=v^{-\frac{m+2}{m-2}}L_{\ig^u}v=w^{-\frac{m+2}{m-2}}L_\ig w.
\]
and this recovers the conformal covariance
\begin{\equ}\label{conformal covariance of Yamabe operator}
	L_{\ig^u}v= u^{-\frac{m+2}{m-2}}L_\ig (uv).
\end{\equ}




Given a metric $\ig(0)=\ig^{u_0}\in[\ig]$, where $u_0\in\cc_+^k$ is such that the Dirac operator $D_{\ig(0)}$ has a simple eigenvalue $\lm(u_0)\neq0$. By virtue of Section \ref{sec:eigenvalues}, we can assume there is an open neighborhood $\cu_0$ of $u_0$ in $\cc_+^k$ such that the simple Dirac eigenvalue $\lm:\cu_0\to(0,+\infty)$, $u\mapsto \lm_u:=\lm(u)$ is well-defined. Then 
\[
\ce:=\Big\{(u,\psi)\in \cu_0\times L^2(M,\mbs(M)) :\, \psi\in\ker\big( D_\ig-\lm_u u^{\frac2{m-2}} \big)  \Big\}  
\]
gives rise to a smooth line bundle over $\cu_0$ equipped with the $(\cdot,\cdot)_2^u$-metric at each $u$. Building upon this, we can evolve $\ig(0)$ through a family of conformal metrics $\ig(t)=\ig^{u(\cdot,t)}$, $t\geq0$, constrained on $\ce$ by considering the flow equation
\begin{\equ}\label{ED-flow}
	\frac{\pa u^{\frac{m+2}{m-2}}}{\pa t}=-\frac{m+2}{m-2}\bigg[ L_{\ig}u-\bigg(\frac{\int_M uL_{\ig}u\,d\vol_{\ig}}{\int_M u^{\frac2{m-2}}|\psi|_\ig^2d\vol_\ig}\bigg) |\psi|_{\ig}^2 u^{\frac{4-m}{m-2}} \bigg]
\end{\equ}
with $(u(\cdot,t),\psi(\cdot,t))\in\ce$. We shall refer this equation as the {\it Einstein-Dirac flow}. 

Evidently, by setting
\[
\psi_u=\frac{\psi}{\sqrt{\int_M u^{\frac2{m-2}}|\psi|_\ig^2d\vol_\ig}} 
\]
for $(u,\psi)\in\ce$, we have $\psi_u$ is $(\cdot,\cdot)_2^u$-normalized in the sense that $(\psi_u,\psi_u)_2^u\equiv1$ and \eqref{ED-flow} can be simplified to the following parabolic type equation
\begin{\equ}\label{the-parabolic-evo}
	\frac{\pa u^{\frac{m+2}{m-2}}}{\pa t}=-\frac{m+2}{m-2}\Big[ L_{\ig}u-\Big(\int_M uL_{\ig}u\,d\vol_{\ig}\Big) |\psi_u|_{\ig}^2 u^{\frac{4-m}{m-2}} \Big].
\end{\equ}

Notice that the map $t\mapsto\ig(t)$ preserves the conformal structure, by denoting
\[
\eta_u(\cdot,t)=-\frac{4}{m-2}\Big[ \scal_{\ig(t)}-\Big(\int_M u(\cdot,t)L_{\ig}u(\cdot,t)\,d\vol_{\ig}\Big) |\psi_u(t)|_{\ig}^2 u(\cdot,t)^{-\frac{2(m-1)}{m-2}} \Big],
\]
we can equivalently transform \eqref{the-parabolic-evo} into 
\[
	\pa_tu=\frac{m-2}4\eta_uu
\]
or the evolution of $\ig$ as
\[
	\pa_t\ig(t)=\eta_u\ig(t).
\]

Clearly, the flow problem \eqref{the-parabolic-evo} leads us to consider a quasilinear parabolic scalar equation, but with an unconventional term involving a concomitant eigenspinor. It will surprise nobody that such an equation still admit short-time (or even long-time) solutions. Nevertheless, this does not belong to the standard theory and requires proof. 

\subsection{The local well-posedness}

We construct now the short-time solution to the Einstein-Dirac flow \eqref{the-parabolic-evo} with initial data  $(u_0,\psi_0)\in\ce$. Our main result in this subsection is as follows. 

\begin{Prop}\label{prop:local-well-posedness}
There exists $T>0$ such that the Einstein-Dirac flow  with above initial data  $(u_0,\psi_0)\in\ce$ has a unique smooth solution in the time interval $[0,T)$.
\end{Prop}

For $l\in\N$ and $T>0$, let us introduce the following Hilbert space
\[
P^l(T):=\Big\{ u:M\times[0,T]\to\R:\, \sum_{2j+k\leq 2l}\int_0^T\|\pa_t^ju\|_{H^k(M)}^2dt<+\infty\Big\}
\]
and equip it with the norm
\[
\|u\|_{P^l(T)}^2=\sum_{2j+k\leq 2l}\int_0^T\|\pa_t^ju\|_{H^k(M)}^2dt.
\]
Clearly, for every $T>0$, there is a natural continuous embedding $P_a^l\hookrightarrow P^l(T)$ by restricting all functions on the time interval $[0,T]$. And, in the study of short time existence result, it will be easier to use the space $P^l(T)$ instead of the weighted spaces $P_a^l$, $a>0$.

Before we consider our flow problem \eqref{the-parabolic-evo}, we revisit the equation \eqref{a linear problem} since it will serve as a linearization of \eqref{the-parabolic-evo}. For later use, for a function $w\in P_a^l$, the current state of $w$ at time $t\in[0,+\infty)$ will be referred to $(t, w(\cdot,t))$ with $w(\cdot,t)\in P_{a,t}^l$, where $P_{a,t}^l$ stands for the function space at time $t$ (and we will simply call it $t$-fiber of $P_a^l$). Then, a crucial requirement for the operator $\cl$ to ensure the short-term existence result is that the current state of $\cl[u]$ should never be affected by the future state of $u$. 
In view of this, let us present the following lemma in the setting of $P^l(T)$ spaces.

\begin{Lem}\label{lem:cl-PT}
Under the assumptions {\rm (A2)} and {\rm (A3)},
for every $l\in\N$ and $T>0$, $\cl$ is a continuous linear operator on $P^l(T)$, i.e., $\cl\in\msl(P^l(T),P^l(T))$.
\end{Lem}
\begin{proof}
We first claim that
\begin{claim*}
	$\cl$ is $t$-fiber preserving on the $P_a^l$ spaces in the sense that $\cl[u](\cdot,t)\in P_{a,t}^l$ depends only on the current state of $u$ at time $t$ for all $u\in P_a^l$.
\end{claim*}
\noindent
And then, we only need to check the boundedness of $\cl$ on $P^l(T)$. Let us assume to the contrary that there exists a $P^l(T)$-bounded sequence of functions $\{u_j\}_{j=1}^\infty$ such that $\|\cl[u_j]\|_{P^l(T)}\to\infty$ as $j\to\infty$.

For each $j\geq1$, let us consider an bounded extension $\tilde u_j\in P_a^l$ of the function $u_j$ in the sense that
\[
\tilde u_j|_{M\times[0,T]}=u_j \quad \text{and} \quad \|\tilde u_j\|_{P_a^l}\leq C\|u_j\|_{P^l(T)}
\]
where $C>0$ is independent of $j$. Then, by the assumption (A3) and the above claim, we have that $\{\cl[\tilde u_j]\}$ is a $P_a^l$-bounded sequence satisfying $\cl[\tilde u_j]|_{M\times[0,T]}=\cl[u_j]$. On the other hand, it follows easily that 
\begin{\equ}\label{extension-bound}
\aligned
\|\cl[\tilde u_j]\|_{P_a^l}^2&=\sum_{2j+k\leq 2l}\int_0^\infty e^{-2at}\|\pa_t^j \cl[\tilde u_j]\|_{H^k(M)}^2dt \\
&\geq C_{a,T}\sum_{2j+k\leq 2l}\int_0^T\|\pa_t^j\cl[\tilde u_j]\|_{H^k(M)}^2dt
=C_{a,T}\|\cl[u_j]\|_{P^l(T)}^2
\endaligned
\end{\equ}
where $C_{a,T}>0$ depends only on $a$ and $T$. And since we have assumed that $\|\cl[u_j]\|_{P^l(T)}\to\infty$ as $j\to\infty$, we can conclude from \eqref{extension-bound} that $\|\cl[\tilde u_j]\|_{P_a^l}\to\infty$, which is absurd. And therefore, $\cl$ must be a bounded linear operator on $P^l(T)$.

\medskip

It remains to check the above claim. And let us point out here that the claim is actually equivalent to the assumption (A2) imposed before. In fact, due to the linearity of $\cl$, it is evident that the $t$-fiber preserving property of $\cl$  implies (A2). And to see the reverse, let us take an arbitrary function $\al\in C^\infty(\R)$ so that $\al(0)=1$ and $\al\equiv0$ on $\R\setminus(-1,1)$, and define $\al_{\tau,\vr}(t)=\al((t-\tau)/\vr)$ for $t\geq0$ with $\tau\in[0,+\infty)$ and $\vr>0$. Then we can infer from (A2) that
\[
\cl[u](\cdot,\tau)=\lim_{\vr\to0}\al_{\tau,\vr}(\tau)\cl[u](\cdot,\tau)=\lim_{\vr\to0}\cl[\al_{\tau,\vr}u](\cdot,\tau).
\]
This, together with the arbitrariness of $\al$, suggests that $\cl[u](\cdot,\tau)$ depends only on the current state $(\tau,u(\cdot,\tau))$. 
\end{proof}


Analogous to Theorem \ref{thm-mff-isomorphism}, we have the following time-localized result in the setting of $P^l(T)$ spaces.

\begin{Lem}\label{lem:mff-isomorphism-PlT}
For every $T>0$ and $l\in\N\cup\{0\}$, the map $\mff$ given by formula \eqref{the map mff} is an isomorphism of $P^{l+1}(T)$ onto $H^{2l+1}(M)\times P^l(T)$.
\end{Lem}
\begin{proof}
With Lemma \ref{lem:cl-PT} readily available, we can establish the continuity of $\mff$ by employing the same reasoning as in Theorem \ref{thm-mff-isomorphism}. Hence, our sole objective is to prove that $\mff$ is an isomorphism.

Given any $u_0\in H^{2l+1}(M)$ and $f\in P^l(T)$, let us consider an extension $\tilde f\in P_a^l$ of the function $f$ and we set $\tilde u\in P_a^{l+1}$ be the solution of the linear problem \eqref{equ-initial-general} in $M\times[0,+\infty)$. Clearly, $u:=\tilde u|_{M\times[0,T]}$ belongs to $P^{l+1}(T)$ and satisfies $\mff[u]=(u_0,f)$ in $M\times[0,T]$. Suppose that $v\in P^{l+1}(T)$ is another function such that $\mff[v]=(u_0,f)$ in $M\times[0,T]$, then we set $w=u-v$. Now, it is easy to see that $w$ satisfies the equation
\[
\pa_t w-\ca\De_\ig w+\cl[w]=0, \quad w(\cdot,0)=0.
\]
Notice that, by \eqref{Garding-inequ}, we can get
\[
\aligned
\int_M w^2(x,t)d\vol_\ig&=2\int_0^t\int_Mw(x,s)\pa_t w(x,s) d\vol_\ig ds \\
&=2\int_0^t\int_M w(x,s)\big[\ca\De_\ig w-\cl[w]\big](x,s)d\vol_\ig ds\\
&\leq -\de \int_0^t\|w(\cdot,s)\|_{H^1(M)}^2 ds + 2\ka\int_0^t |w(\cdot,s)|_2^2 ds \\
&\leq 2\ka\int_0^t \int_M w^2(x,s)d\vol_\ig ds
\endaligned
\]
as $w(\cdot,t)\in W^{2,2}(M)$ for almost every $t\in[0,T]$. Then, by Gronwall's inequality, we have that $\int_M w^2(x,t)d\vol_\ig$ is zero for every $t\in[0,T]$. It follows that $w$ is identically zero, and hence $u$ and $v$ must coincide.

Since $\mff:P^{l+1}(T)\to H^{2l+1}(M)\times P^l(T)$ is a continuous bijection, we can conclude from the open mapping theorem that $\mff$ is an isomorphism.
\end{proof}

\begin{Rem}\label{rem:solu smooth}
If both $u_0$ and $f$ are smooth functions, then the unique solution $u$ of Eq. \eqref{equ-initial-general} is guaranteed to be in the space $\cap_{l\geq1}P^l(T)$. By the Parabolic Sobolev embeddings (see for instance Proposition \ref{Parabolic-Sobolev-Embeddings} in the Appendix), for any $k\in\N$, we can always find a large $l$ so that $P^l(T)$ continuously embeds into $C^k(M\times[0,T])$. And hence we can conclude that the solution $u$ actually belongs to $C^\infty(M\times[0,T])$ provided that $u_0$ and $f$ are smooth.
\end{Rem}

Now we are nearing the point where we can provide the short time existence result for the Einstein-Dirac flow \eqref{the-parabolic-evo}. Before we start the proof, let us denote $P^l_+(T)$ as the positive cone in $P^l(T)$, consisting of all $u\in P^l(T)$ such that $u>0$. We also introduce a subset of $P_+^l(T)$ as
\[
P_{+,1}^l(T):=\big\{ u\in P_+^l(T):\,u(\cdot,0) \in \cu_{\ig,n_0}^1  \big\}
\]
where $\cu_{\ig,n_0}^1$ is given by \eqref{cu-g-n-k} with $n_0>1$ being initially fixed. We will now consider the a quasilinear elliptic operator  given by
\[
\cq[u]=c_m u^{-\frac4{m-2}}\De_{\ig} u - \scal_{\ig} u^{\frac{m-6}{m-2}}+\Big(\int_M u L_{\ig}u \,d\vol_{\ig}\Big)|\psi_u|_{\ig}^2 u ^{-\frac m{m-2}},
\]
where $c_m=\frac{4(m-1)}{m-2}$ and $\psi_u(t)\in L^2(M,\mbs(M))$ is the normalized eigenspinor associated to the simple Dirac eigenvalue $\lm_u(t)\neq0$ with $\lm_u(0)\in[-n_0,0)\cup(0,n_0]$, which exists at least for a small time interval $[0,T_u]$. Let us mention that, due to the parabolic Sobolev embeddings (see for instance \cite[Proposition 4.1]{MM}), we have $P^l(T)\hookrightarrow C^1(M\times[0,T])$ for $l>\frac{m+6}4$. Then, under the hypothesis of Lemma \ref{lem:open-dense}, we can see that $P_{+,1}^l(T)$ is an open dense subset of $P_+^l(T)$.  We also emphasize that the set $P_{+,1}^l(T)$ can be regarded as a collection of curves of functions with their starting point in $\cu_{\ig,n_0}^1$. It is possible that for certain curves in this collection, the Dirac operator associated with the conformal metric $\ig^{u(\cdot,t_*)}$ at a particular time $t_*>0$ may not possess a simple eigenvalue. However, thanks to Lemma \ref{lem:Ck-perturbed-eigenvalue}, we can always find a $P^l(T)$-neighborhood $U$ of a given function $u\in P_{+,1}^l(T)$ (by setting $T$ smaller if necessary) so that $\cq$ is well-defined in the sense that $t\mapsto(\lm_v(t),\psi_v(t))$ exists (at least continuously) for all $v\in U$ on $[0,T]$. And we shall call $U$ the {\it $\cq$-admissible neighborhood} of $u$. In order to study the short time existence result of Eq. \eqref{the-parabolic-evo}, for a given $u\in P_{+,1}^l(T)$ and its $\cq$-admissible neighborhood $U\subset P_{+,1}^l(T)$, let us consider 
\begin{\equ}\label{the-map-msf}
\msf[v]=(v_0,\pa_tv-\cq_*[v]), \quad \forall v\in U
\end{\equ}
where $v_0:=v(\cdot,0)$ and 
\[
\cq_*[v]:=c_m v^{-\frac4{m-2}}\De_{\ig} v - \scal_{\ig} v^{\frac{m-6}{m-2}}+\vartheta(t)\Big(\int_M v L_{\ig}v \,d\vol_{\ig}\Big)|\psi_v|_{\ig}^2 v ^{-\frac m{m-2}},
\]
with $\vartheta\in C^\infty([0,+\infty))$ satisfying
\[
\vartheta(t)=\begin{cases}
	1, & t\leq T/2 \\
	0, & t\geq 3T/4
\end{cases}.
\]

Once a $\cq$-admissible neighborhood $U$ is fixed, one of our key ingredient here is the observation that $\msf: U\to H^{2l-1}(M)\times P^{l-1}(T)$ is of class $C^1$ provided that $l$ is suitably large (say $l>(m+8)/2$). In particular, the Fr\'echet derivative of $\msf$ at $v\in U$ can be written in the form
\[
D\msf[v](w)=\frac{d}{d\tau}\msf[v+\tau w]\Big|_{\tau=0}=\big(w_0,  \pa_tw - c_mv^{-\frac4{m-2}}\De_\ig w + \cl_v[w] \big)
\]
for $w\in P^l(T)$, where $\cl_v \in\msl(P^l(T),P^l(T))$ depends continuously on $v$ (the explicit formulation of $\cl_v$ will be given later). Notice that we have inserted a cut-off function $\vartheta$ in the formulation of $\cq_*$,  we can actually extend  $\cq_*$ to act on any positive function $\tilde v\in P_a^l$, $a>0$, such that $\tilde v|_{M\times[0,T]}=v\in U$. And hence the induced operator $\cl_{\tilde v}$ exists as well for such $\tilde v$. By taking a cut-off function $\rho\in C^\infty([0,+\infty))$ so that $\rho(t)=1$ for $t\leq T$ and $\rho(t)=0$ for $t\geq 2T$, we can introduce an extension of $D\msf[v]$ as
\begin{\equ}\label{the extended map}
\mff_{\tilde v}[w]=\big( w_0, \pa_t w- (c_m\rho \tilde v^{-\frac4{m-2}} +1- \rho) \De_\ig w+\rho  \cl_{\tilde v}[w]\big) \quad \text{for } w\in P_a^l.
\end{\equ}
And then we have $\mff_{\tilde v}|_{P^l(T)}=D\msf[v]$. In case $\tilde v$ is a smooth function, the above expression will make it legitimate to apply Theorem \ref{thm-mff-isomorphism} and Lemma \ref{lem:mff-isomorphism-PlT}  to $\mff_{\tilde v}$ and $D\msf[v]$ respectively. All these will be verified in Appendix \ref{A. Technical results}.

The proof of Proposition \ref{prop:local-well-posedness} will be divided into two classical steps: the local existence and the energy inequality. 

\begin{proof}[Proof of Proposition \ref{prop:local-well-posedness}: local existence]
Given $u_0\in \cap_{k\in\N}\, \cu_{\ig,n_0}^k$, let us consider an  auxiliary function
\begin{\equ}\label{auxiliary function w}
w_*(x,t)=\sum_{i=0}^{l-1}\frac{a_i(x)t^i}{i!}
\end{\equ}
with $l\in\N$ being necessarily large, $a_0=u_0$ and $a_1,\dots,a_{l-1}\in C^\infty(M)$ to be determined later. Notice that $a_0>0$, by virtue of Lemma \ref{lem:Ck-perturbed-eigenvalue}, we can always find $T>0$ small such that $w_*\in P_{+,1}^l(T)$ and the eigenpair $(\lm_{w_*}(t),\psi_{w_*}(t))$ exists on $[0,T]$. Furthermore, we can take $u_*\in P^l(T)$ to be the unique solution of the linear problem
\[
\left\{
\aligned
&\pa_t u=c_m w_*^{-\frac4{m-2}}\De_{\ig} u-\scal_{\ig} w_*^{\frac{m-6}{m-2}}+\vartheta\cdot\Big(\int_M w_* L_{\ig}w_* \,d\vol_{\ig}\Big)|\psi_{w_*}|_{\ig}^2 w_*^{-\frac m{m-2}} \\
& u(x,0)= u_0(x).
\endaligned
\right.
\]
Such solution exists by Lemma \ref{lem:mff-isomorphism-PlT} and it is smooth by Remark \ref{rem:solu smooth}, as $u_0$ and $w_*$ are smooth (thus the eigenspinor $\psi_{w_*}$ is also smooth).

By narrowing the time interval $[0,T]$ (if necessary), we have $u_*\in P^l_{+,1}(T)$ and
\[
\msf[u_*]=(u_0, \pa_t u_*-\cq_*[u_*])=:(u_0,f_*)
\]
where $f_*$ has the explicit expression
\[
\aligned
f_*&=c_m\big( w_*^{-\frac4{m-2}} - u_*^{-\frac4{m-2}} \big)\De_{\ig} u_*-\scal_{\ig}\Big( w_*^{\frac{m-6}{m-2}}-u_*^{\frac{m-6}{m-2}} \Big) \\
&\qquad +\vartheta\cdot\Big(\int_M w_* L_{\ig}w_* \,d\vol_{\ig}\Big)|\psi_{w_*}|_{\ig}^2 w_*^{-\frac m{m-2}} \\
&\qquad -\vartheta\cdot\Big(\int_M u_* L_{\ig}u_* \,d\vol_{\ig}\Big)|\psi_{u_*}|_{\ig}^2 u_*^{-\frac m{m-2}}.
\endaligned
\]

If we compute the Fr\'echet derivative of the map $\msf$ at $u_*\in C^\infty(M\times[0,T])$, in the ``direction" $w\in P^l(T)$, we obtain
\[
D\msf[u_*](w)=\big(w_0, \pa_t w-c_m u_*^{-\frac4{m-2}}\De_{\ig} w +\cl_{u_*}[w]  \big)
\]
where $w_0:=w(\cdot,0)$ and
\[
\aligned
\cl_{u_*}[w] &= \frac{4c_m}{m-2}u_*^{-\frac{m+2}{m-2}}w\De_{\ig} u_*-\frac{m-6}{m-2}\scal_{\ig}u_*^{-\frac{4}{m-2}}w \\
&\qquad +2\vartheta\cdot\Big(\int_M w L_{\ig}u_* \,d\vol_{\ig}\Big)|\psi_{u_*}|_{\ig}^2 u_*^{-\frac m{m-2}} \\
&\qquad +2\vartheta\cdot\Big(\int_M u_* L_{\ig}u_* \,d\vol_{\ig}\Big)\real(\psi_{u_*},\phi_{u_*,w})_{\ig} u_*^{-\frac m{m-2}} \\
&\qquad -\frac{m}{m-2}\vartheta\cdot\Big(\int_M u_*L_{\ig}u_* \,d\vol_{\ig}\Big)|\psi_{u_*}|_{\ig}^2 u_*^{-\frac {2m-2}{m-2}} w
\endaligned
\]
with
\[
\aligned
\phi_{u_*,w}(t)&=\frac{\ka_{u_*,w}(t)}2\psi_{u_*}(t) \\
&\qquad -\frac{2\lm_{u_*}(t)}{m-2}\big( u_*(\cdot,t)^{-\frac2{m-2}}D_\ig-\lm_{u_*}(t) \big)^{-1}\circ(I-P_{\lm_{u_*}(t)})\big( u_*(\cdot,t)^{-1} w(\cdot,t)\psi_{u_*}(t) \big)
\endaligned
\]
and
\[
\ka_{u_*,w}(t)=-\frac2{m-2}\int_M u_*(\cdot,t)^{\frac{4-m}{m-2}}w(\cdot,t)|\psi_{u_*}(t)|_\ig^2d\vol_\ig.
\]
And we can find that $D\msf[u_*](w)=(z,f)\in H^{2l-1}(M)\times P^{l-1}(T)$ is equivalent to the linear problem
\begin{\equ}\label{DFu}
\left\{
\aligned
&\pa_t w-c_m u_*^{-\frac4{m-2}}\De_\ig w + \cl_{u_*}[w]= f \\
&w(x,0)=z(x)
\endaligned 
\right.
\end{\equ}
which coincides with our model problem \eqref{a linear problem} on $M\times[0,T]$. 

By applying Lemma \ref{lem:mff-isomorphism-PlT}, we can see that $D\msf[u_*]: P^l(T)\to H^{2l-1}(M)\times P^{l-1}(T)$ is an isomorphism. And, in other words, there exists a unique solution $w_{z,f}$ of Eq. \eqref{DFu} for every pair $(z,f)\in H^{2l-1}(M)\times P^{l-1}(T)$. Now, by virtue of the inverse function theorem, we find the map $\msf$ is a diffeomorphism from a neighborhood $U\subset P_{+,1}^l(T)$ of $u_*$ onto a neighborhood $V\subset H^{2l-1}(M)\times P^{l-1}(T)$ of $(u_0,f_*)$. Particularly, $U$ can be chosen as a $\cq$-admissible neighborhood of $u_*$.

Next, we can consider a sequence of time-shifted functions $f_{*,k}:M\times[0,T]\to\R$ given by
\[
f_{*,k}(x,t)=\begin{cases}
	0 & \text{if } 0\leq t<1/k ,\\
	f_*(x,t-1/k) & \text{if } 1/k\leq t\leq T.
\end{cases}
\]
Assuming for the moment that 
\begin{\equ}\label{f--0}
	\pa_t^j f_*|_{t=0}=0 \text{ and }  \nabla_\ig^p\pa_t^j f_*|_{t=0}=0 \text{ for any } j=0,1,\dots,l-1 \text{ and } p\in\N.
\end{\equ}
We can then directly deduce $\nabla_\ig^p\pa_t^j f_{*,k}\in C(M\times[0,T])$ for every $j=0,1,\dots,l-1$ and $p\in\N$. Moreover, it follows easily that
\[
\nabla_\ig^p\pa_t^j f_{*,k}\to \nabla_\ig^p\pa_t^j f_{*} \text{ in } L^2(M\times[0,T]) \text{ for } j=0,1,\dots,l-1 \text{ and } 2j+p\leq 2(l-1),
\]
and hence $f_{*,k}\to f_*$ in $P^{l-1}(T)$. From this point of view, there is $k_0\in\N$ such that $(u_0,f_{*,k_0})$ lies in the neighborhood $V$ of $\msf[u_*]=(u_0,f_*)$ and $f_{*,k_0}=0$ on $M\times[0,1/k_0]$. Since $\msf$ is a diffeomorphism between $U$ and $V$, we can find a function $u_{*,k_0}\in U$ such that $\msf[u_{*,k_0}]=(u_0,f_{*,k_0})$. Clearly, such $u_{*,k_0}\in P_{+,1}^l(1/k_0)$ will be a solution of the Einstein-Dirac flow \eqref{the-parabolic-evo} on $M\times[0,1/k_0]$. By taking $T_0<1/k_0$, the parabolic regularity implies that $u_{*,k_0}\in C^\infty(M\times[0,T_0])$.

It remains to verify \eqref{f--0}, which is strongly depending on the choice of those functions $a_1,\dots,a_{l-1}$ in \eqref{auxiliary function w}. Notice that we have set $a_0=u_0\in\cap_{k\in\N}\,\cu_{\ig,n_0}^k$, let us defined $a_i$ via the recurrence formula 
\[
a_i=\pa_t^{i-1}\Big[ c_m w_*^{-\frac4{m-2}}\De_{\ig} u_*-\scal_{\ig} w_*^{\frac{m-6}{m-2}}+\vartheta\cdot\Big(\int_M w_* L_{\ig}w_* \,d\vol_{\ig}\Big)|\psi_{w_*}|_{\ig}^2 w_*^{-\frac m{m-2}} \Big]\Big|_{t=0}
\] 
for $i\geq1$. Observe that the right-hand side above contains time-derivatives at $t=0$ of $w_*$, $\nabla_\ig w_*$, $L_{\ig}w_*$ and $\De_{\ig} u_*$ only up to order $i-1$, hence it depends only on the functions $a_0,\dots, a_{i-1}$. In this way, we get the precise details of the functions $a_1,\dots, a_{l-1}\in C^\infty(M)$. Following on this clue, we point out that $a_i=\pa_t^iw_*|_{t=0}=\pa_t^i u_*|_{t=0}$ for $i=0,1,\dots, l-1$. And hence $f_*\in C^\infty(M\times[0,T])$ satisfies \eqref{f--0} as was required.
\end{proof}

Having disposed of the local existence result, we will now turn to consider the uniqueness of a solution and the continuous dependence on its initial data. 
Actually, we shall mention here that the proof above gives us more. In fact, once the uniqueness is settled, the solution of the Einstein-Dirac flow \eqref{the-parabolic-evo} with initial data $(u_0,\lm_0,\psi_0)\in \cap_{k\in\N}\,\cu_{\ig,n_0}^k\times(-n_0,n_0)\times \ker\big(D_\ig-\lm_0u_0^{2/(m-2)}\big)$ can be characterized by $u=\msf^{-1}(u_0,0)$ restricted on $M\times[0,T_0]$. Assuming $u_{j,0}\to u_0$ in $C^\infty(M)$ as $j\to\infty$, we also have $u_{j,0}\to u_0$ in $H^{2l-1}(M)$ for all $l\in\N$, and hence for $j$ large enough we have $u_j=\msf^{-1}(u_{j,0},0)\in P_{+,1}^l(T_0)$ is the unique solution with initial data $(u_{j,0},\lm_{j,0},\psi_{j,0})$. By recalling the continuity of perturbed Dirac eigenvalue problem mentioned in Section \ref{sec:eigenvalues}, the eigenpair $(\lm_{j,0},\psi_{j,0})$ can be chosen so that $\lm_{j,0}\to\lm_0$ and $\psi_{j,0}\to\psi_0$. Therefore, we have $u_j\to u$ in $P^l(T_0)$-topology, which is nothing but the continuous dependence on the initial data. In what follows, let us verify the uniqueness result of the Einstein-Dirac flow.

\begin{proof}[Proof of Proposition \ref{prop:local-well-posedness}: energy inequality]
Suppose that we have two $P_{+,1}^l(T_0)$-solutions $u$ and $v$ of \eqref{the-parabolic-evo} on $M\times[0,T_0]$ with $u(\cdot,0)=v(\cdot,0)=u_0$, $\lm_u(0)=\lm_v(0)$ and $\psi_u(0)=\psi_v(0)$, where $l$ is suitably large. By setting $w=u-v$, we find that
\begin{eqnarray*}
\frac{d}{dt}\int_M|\nabla_\ig w|^2d\vol_\ig&=&2\int_M\nabla_\ig w\cdot\nabla_\ig\pa_tw\,d\vol_\ig\\
&=&2\int_M\nabla_\ig w\cdot\nabla_\ig\big( \cq[u]-\cq[v] \big)d\vol_\ig \\
&=&-2\int_M\De_\ig w\cdot\big( \cq[u]-\cq[v] \big)d\vol_\ig \\
&=&-2c_m\int_M u^{-\frac4{m-2}}|\De_\ig w|^2  d\vol_\ig \\
&&\qquad - 2c_m\int_M\De_\ig w\cdot\big( \big( u^{-\frac4{m-2}}-v^{-\frac4{m-2}} \big)\De_\ig v \big)d\vol_\ig \\
&&\qquad +2\int_M\De_\ig w\cdot\scal_{\ig}\cdot\big(  u^{\frac{m-6}{m-2}}- v^{\frac{m-6}{m-2}} \big) d\vol_\ig\\
&&\qquad -2\int_M \De_\ig w\cdot \big( \cs[u]-\cs[v] \big)d\vol_\ig,
\end{eqnarray*}
where
\[
\cs[z]=\Big( \int_M zL_\ig z\,d\vol_\ig \Big)|\psi_z|_{\ig}^2 z^{-\frac{m}{m-2}}
\]
for $z=u$ or $v$.

Here we need to be careful because the well-definedness of $\cs[u]$ and $\cs[v]$ do not imply the existence of $\cs[su+(1-s)v]$ for all $s\in[0,1]$ on $M\times[0,T_0]$. However, since we have assumed $u(\cdot,0)=v(\cdot,0)=u_0$, it is possible to find $\tau\in(0,T_0]$ and a $\cq$-admissible neighborhood $U\subset P_{+,1}^l(\tau)$ of $u|_{M\times[0,\tau]}$ such that $(su+(1-s)v)|_{M\times[0,\tau]}\in U$ for all $s\in[0,1]$. Indeed, we can first find a $C^1$-neighborhood $\co$ of $u_0$ so that $\co\subset\cu_{\ig,n_0}^1$ and $u(\cdot,t), v(\cdot,t)\in\co$ for all $t\in[0,\tau]$. Then, by Lemma \ref{lem:Ck-perturbed-eigenvalue} and the parabolic Sobolev embedding Proposition \ref{Parabolic-Sobolev-Embeddings}, we have $\cs$ can be defined on any $t$-parameterized trajectory of functions in $\co$. Notice that we have $(su+(1-s)v)(\cdot,t)\in \co$ for all $t\in[0,\tau]$ and $s\in[0,1]$, we can thus ensure $(su+(1-s)v)|_{M\times[0,\tau]}$ lies in a $\cq$-admissible neighborhood of $u|_{M\times[0,\tau]}$. And therefore, we can consider the map  $s\mapsto \cs[su+(1-s)v]$ on $M\times[0,\tau]$ for all $s\in[0,1]$.

Let's go back to the previous computation. By setting $l$ large enough (so that Lemma \ref{lem:msf-C1} applies), the positiveness and boundedness of $u$ and $v$ imply that 
\[
\int_M u^{-\frac4{m-2}}|\De_\ig w|^2  d\vol_\ig\geq \de\int_M |\De_\ig w|^2 d\vol_\ig,
\]
\[
\aligned
\int_M\De_\ig w\cdot\big( \big( u^{-\frac4{m-2}}-v^{-\frac4{m-2}} \big)\De_\ig v \big)d\vol_\ig &\leq C\int_M |\De_\ig w| |u-v| d\vol_\ig \\
&=C\int_M |\De_\ig w| |w| d\vol_\ig \\
&\leq \vr\int_M|\De_\ig w|^2d\vol_\ig + C_\vr \int_M|w|^2d\vol_\ig,
\endaligned
\]
\[
\int_M\De_\ig w\cdot\scal_{\ig}\cdot\big(  u^{\frac{m-6}{m-2}}- v^{\frac{m-6}{m-2}} \big) d\vol_\ig\leq \vr\int_M|\De_\ig w|^2d\vol_\ig + C_\vr \int_M|w|^2d\vol_\ig,
\]
and
\[
\aligned
\int_M \De_\ig w\cdot \big( \cs[u]-\cs[v] \big)d\vol_\ig&\leq \vr\int_M|\De_\ig w|^2d\vol_\ig + C_\vr \int_M \big| \cs[u]-\cs[v] \big|^2d\vol_\ig\\
&\leq \vr\int_M|\De_\ig w|^2d\vol_\ig + C_\vr \int_M \big| D\cs[v+\theta w](w) \big|^2d\vol_\ig\\
&\leq \vr\int_M|\De_\ig w|^2d\vol_\ig + C_\vr \int_M \big( |\nabla_\ig w|^2+ |w|^2 \big) d\vol_\ig,
\endaligned
\]
where in the last estimate we have designated $D\cs$ as the Fr\'echet derivative of $\cs:P_{+,1}^l(\tau)\supset U\to P^l(\tau)$ (which exists due to Lemma \ref{lem:msf-C1} and satisfies assumption (A1) as is indicated by Corollary \ref{lem:A1-A3}). By fixing $\vr$ small, we get
\begin{\equ}\label{dnablaw}
\frac{d}{dt}\int_M|\nabla_\ig w|^2d\vol_\ig\leq -\frac\de2 \int_M |\De_\ig w|^2 d\vol_\ig+ C \int_M \big( |\nabla_\ig w|^2+ |w|^2 \big) d\vol_\ig.
\end{\equ}
Next, using similar idea as above, we can compute 
\begin{eqnarray*}
\frac{d}{dt}\int_M|w|^2d\vol_\ig&=&2\int_M w\cdot\big( \cq[u]-\cq[v] \big)d\vol_\ig \\
&=&2c_m\int_M u^{-\frac4{m-2}}w\De_\ig w \,d\vol_\ig \\
&&\quad +2c_m\int_M w\cdot\big( \big( u^{-\frac4{m-2}}-v^{-\frac4{m-2}} \big)\De_\ig v \big)d\vol_\ig \\
&&\quad -2\int_M w\cdot\scal_{\ig}\cdot\big(  u^{\frac{m-6}{m-2}}- v^{\frac{m-6}{m-2}} \big) d\vol_\ig\\
&&\quad +2\int_M w\cdot\big( \cs[u]-\cs[v] \big)d\vol_\ig \\
&&\leq  \vr \int_M|\De_\ig w|^2d\vol_\ig + C_\vr\int_M\big( |\nabla_\ig w|^2+ |w|^2 \big) d\vol_\ig.
\end{eqnarray*}
By fixing $\vr$ small and putting the above estimate together with \eqref{dnablaw}, we obtain
\begin{\equ}\label{dH1norm}
\frac{d}{dt}\int_M\big( |\nabla_\ig w|^2+ |w|^2 \big) d\vol_\ig \leq C \int_M\big( |\nabla_\ig w|^2+ |w|^2 \big) d\vol_\ig
\end{\equ}
for some constant $C>0$ on the time interval $[0,\tau]$.

Finally, from \eqref{dH1norm}, we can conclude that the quantity $\int_M\big( |\nabla_\ig w|^2+ |w|^2 \big) d\vol_\ig$ is identically zero on $[0,\tau]$ since we have $w(\cdot,0)=0$. Moreover, by repeating the above argument, it follows that if  $\int_M\big( |\nabla_\ig w|^2+ |w|^2 \big) d\vol_\ig$ is zero at some time $t_0$, then it must be zero for every time $t\in[t_0,T_0]$. The proof is hereby completed.
\end{proof}

%
%
	

\appendix
\section{Appendix}

\subsection{Parabolic Sobolev embeddings}

We present here some embedding results that were used in the context. For a proof of the following proposition, we refer the readers to \cite[Proposition 4.1]{MM}.

\begin{Prop}\label{Parabolic-Sobolev-Embeddings}
Let $u\in P^{l}(T)$. Then, for $j,k\in\N$ with $2j+k\leq 2l$, we have
\[
\|\pa_t^j\nabla^ku\|_{L^q(M\times[0,T])}\leq C \|u\|_{P^l(T)} \quad \text{if } \frac1q=\frac12-\frac{2l-k-2j}{m+2}>0;
\]
\[
\|\pa_t^j\nabla^ku\|_{L^q(M\times[0,T])}\leq C \|u\|_{P^l(T)} \quad \text{if } \frac1q=\frac12-\frac{2l-k-2j}{m+2}=0 \text{ and } 1\leq q<+\infty;
\]
the function $\pa_t^j\nabla ^k u$ is continuous and 
\[
\|\pa_t^j\nabla^ku\|_{C(M\times[0,T])}\leq C\|u\|_{P^l(T)}\quad \text{if } \frac1q=\frac12-\frac{2l-k-2j}{m+2}<0,
\]
where the constant $C>0$ is independent of $u$.
\end{Prop}

\subsection{Regularity theory for the linear problem \eqref{equ-initial-general}}\label{A. linear regualrity}

In the following discussion, we will demonstrate that a weak solution $u$ of equation \eqref{equ-initial-general} possesses higher-order derivatives by obtaining estimates for its difference quotients.

\begin{Lem}\label{lem:regularity2}
	If $u\in HH_a$ is a weak solution of \eqref{equ-initial-general} with $f\in LH_a^0$ and $u_0\in H^1(M)$ and $a>0$ is suitably large, then $u\in LH_a^2$ and the following estimate holds
	\[
	\|u\|_{LH_a^2}^2\leq C\big( \|u_0\|_{H^1(M)}^2+\|f\|_{LH_a^0}^2 \big),
	\]
	where $C>0$ depends only on $\ca,\cl,\nabla_\ig\ca$ and the manifold $M$.
\end{Lem}
\begin{proof}
	For $r>0$, let us denote $B_r(0)\subset\R^m$ the open ball of radius $r$ centered at the origin.
	Since $M$ is compact, we can choose $r<inj(M)/3$ (where $inj(M)>0$ represents the injectivity radius of $M$) and $\{(U_i,\Psi_i)\}_{i=1}^N$ as a family of smooth local coordinate system on $M$ such that $\Psi_i: B_r(0)\to U_i\subset M$ is a diffeomorphism and $\cup_{i=1}^NU_i=M$. Let $\eta$ be a smooth cut-off function on $\R^m$ satisfying
	\[
	\eta(x)=\begin{cases}
		1 & x\in B_r(0), \\
		0 & x\in \R^m\setminus B_{2r}(0).
	\end{cases}
	\]
	We also denote $U_i^*=\Psi_i(B_{3r}(0))$.
	
	Given some $k\in\N$, let us assume that we already have the estimate for $u$ in $LH_a^{l+1}$, for each $l=0,\dots,k-1$, i.e.,
	\[
	\|u\|_{LH_a^{l+1}}^2\leq C\big( \|u\|_{H^1(M)}^2+\|f\|_{LH_a^0}^2 \big).
	\]
	Evidently, Lemma \ref{lem:regularity1} suggests the above estimate holds for $k=1$. To proceed, for $h\in\R\setminus\{0\}$, let us set $\phi=\varDelta_{-h}^k(\eta^4\varDelta_h^k u)$ in $U_i^*$, where $\eta$ and the finite difference operator 
	\[
	\varDelta_h w = \frac{w(x+he_s)-w(x)}h
	\]
	are lifted to $M$ using the coordinate map $\Psi_i$ and $\varDelta_h^k$ means acting $\varDelta_h$ for $k$ times. In the above formulation, $e_s$ is an arbitrarily fixed element of the local frame $\{e_1,\dots,e_m\}$. Outside $U_i^*$ we simply extend $\phi$ to be zero. At this moment, although we don't have uniform estimate in $h$ for the norms of $\phi$, the function is at least as regular as $u$.
	
	Testing \eqref{equ-initial-general} with $\phi$, we obtain
	\begin{\equ}\label{test}
		\inp{\pa_tu}{\phi}_{LH_a^0}+\int_0^\infty e^{-2at}\msa_t(u,\phi)dt=\inp{f}{\phi}_{LH_a^0}.
	\end{\equ}
	Shifting difference operators with the discrete analogue of partial integration in space, we soon deduce that
	\begin{\equ}\label{expand1}
		\aligned
		&\inp{\pa_t(\eta^2\varDelta_h^k u)}{\eta^2\varDelta_h^k u}_{LH_a^0}-\int_0^\infty e^{-2at}\int_M\ca^*_k\De_\ig(\varDelta_h^ku)\eta^4\varDelta_h^k u \, d\vol_\ig dt \\
		&\quad +(-1)^k\inp{\cl[u]}{\phi}_{LH_a^0}=(-1)^k\inp{f}{\phi}_{LH_a^0} + E
		\endaligned
	\end{\equ}
	where we have set
	\[
	E=\sum_{j=1}^k \begin{pmatrix}
		k\\
		j
	\end{pmatrix}\int_0^\infty e^{-2at}\int_M \varDelta_h^j\ca_{k-j}^* \De_\ig(\varDelta_h^{k-j}u)\eta^4\varDelta_h^k u\, d\vol_\ig dt.
	\]
	Additional explanation may be necessary: the function $\ca_j^*(\cdot,\cdot)$, $j=0,\dots,k$, which is given by $\ca_j^*(x,t)=\ca(x+jhe_s,t)$, arises through an application of the discretised product rule
	\[
	\varDelta_h(wz)(x)=w(x+he_s)\varDelta_h z(x)+z(x)\varDelta_h w(x).
	\]
	Since we have assumed that $\ca$ is smooth, the functions $\ca_0^*,\dots,\ca_k^*$ share all pointwise properties.
	
	Focusing on $E$, we have
	\begin{\equ}\label{E1}
		\aligned
		|E|&\leq C\int_0^\infty e^{-2at}\int_M \sum_{j=1}^k|\nabla_\ig(\varDelta_h^{k-j}u)|\big( |\nabla_\ig(\eta^4\varDelta_h^k u)|+ |\eta^4\varDelta_h^k u| \big)d\vol_\ig dt \\[0.3em]
		&\leq C_\vr\|u\|_{LH_a^k}^2+\vr \|\eta^4\nabla_\ig(\varDelta_h^k u)\|_{LH_a^0}^2,
		\endaligned
	\end{\equ}
	where we have used \cite[Lemma 7.23]{GT} in the last inequality, $\vr>0$ is a small constant whose value will be given later and $C_\vr>0$ depends on $\vr$ and the $L^\infty$-norm of derivatives of the coefficients $\ca_1^*,\dots,\ca_k^*$ (they are also controlled by the corresponding norms of $\ca$). Regarding the remaining terms in \eqref{expand1}, we can address the first term on the left using \eqref{identity1}, that is
	\begin{\equ}\label{E2}
		\inp{\pa_t(\eta^2\varDelta_h^k u)}{\eta^2\varDelta_h^k u}_{LH_a^0}=a\|\eta^2\varDelta_h^k u\|_{LH_a^0}^2-\frac12|\eta^2\varDelta_h^k u_0|_2^2.
	\end{\equ}
	And for the elliptic part, we have
	\[
	\aligned
	E^*&:=-\int_0^\infty e^{-2at}\int_M\ca^*_k\De_\ig(\varDelta_h^ku)\eta^4\varDelta_h^k u \, d\vol_\ig dt \\
	&=\int_0^\infty e^{-2at}\int_M \big(\eta^4\ca_k^*\nabla_\ig(\varDelta_h^k u)\cdot\nabla_\ig(\varDelta_h^k u)+\nabla_\ig(\eta^4\ca_k^*)\cdot\nabla_\ig(\varDelta_h^k u)\varDelta_h^ku \big) d\vol_\ig dt \\
	&\geq \int_0^\infty e^{-2at}\int_M \big(\eta^4\ca_k^*\nabla_\ig(\varDelta_h^k u)\cdot\nabla_\ig(\varDelta_h^k u)- C\eta^2|\nabla_\ig(\varDelta_h^k u)||\varDelta_h^ku|\big) d\vol_\ig dt ,
	\endaligned
	\]
	where $C>0$ depends only on the gradient of $\eta\ca_k^*$.
	And so, using the ellipticity of $\ca_k^*$ and Young's inequality, we find
	\begin{\equ}\label{E3}
		\aligned
		E^*&\geq \de\|\eta^2\nabla_\ig(\varDelta_h^k u)\|_{LH_a^0}^2-\frac\de4\|\eta^2\nabla_\ig(\varDelta_h^k u)\|_{LH_a^0}^2- C_\de\|\varDelta_h^k u\|_{LH_a^0}^2 \\
		&\geq\frac{3\de}4\|\eta^2\nabla_\ig(\varDelta_h^k u)\|_{LH_a^0}^2- C_\de\| u\|_{LH_a^k}^2,
		\endaligned
	\end{\equ}
	where we have used \cite[Lemma 7.23]{GT} again in the last inequality. Now, let us consider the two remaining terms in \eqref{expand1}. Actually, the estimates can be easily obtained if we now explicitly take $k=1$:
	\begin{\equ}\label{E4}
		|\inp{\cl[u]}{\phi}_{LH_a^0}|\leq C_\cl \|u\|_{LH_a^1} \|\eta^4\varDelta_h u\|_{LH_a^1}
	\end{\equ}
	where $C_\cl>0$ comes from the assumption 
	(A1) for $\cl$, and meanwhile,
	\begin{\equ}\label{E5}
		\aligned
		|\inp{f}{\phi}_{LH_a^0}|&\leq \|f\|_{LH_a^0}\|\varDelta_{-h}(\eta^4\varDelta_h u)\|_{LH_a^0} \\
		&\leq \|f\|_{LH_a^0}\|\eta^4\varDelta_h u\|_{LH_a^1}.
		\endaligned
	\end{\equ}
	Combining \eqref{E1}-\eqref{E5}, we have
	\[
	\aligned
	\frac{3\de}4\|\eta^2\nabla_\ig(\varDelta_h u)\|_{LH_a^0}^2- C_\de\|u\|_{LH_a^1}^2&\leq C_\vr\|u\|_{LH_a^1}^2+\vr\|\eta^2\nabla_\ig(\varDelta_h u)\|_{LH_a^0}^2 + \frac12|\eta^2\varDelta_h u_0|_2^2 \\
	&\qquad +C_\cl \|u\|_{LH_a^1} \|\eta^4\varDelta_h u\|_{LH_a^1}+\|f\|_{LH_a^0}\|\eta^4\varDelta_h u\|_{LH_a^1}.
	\endaligned
	\]
	Thus, by choosing $\vr\leq \de/4$, invoking Lemma \ref{lem:regularity1} and recalling the hypothesis that $u_0\in H^1(M)$, we have
	\[
	\aligned
	\|\eta^2\nabla_\ig(\varDelta_h u)\|_{LH_a^0}^2&\leq C\big( \|u_0\|_{H^1(M)}^2+ \|u\|_{LH_a^1}^2+\|f\|_{LH_a^0}^2 \big)\\
	&\leq C\big( \|u_0\|_{H^1(M)}^2+\|f\|_{LH_a^0}^2 \big),
	\endaligned
	\]
	where $C>0$ depends only on $\ca,\cl$ and $\nabla_\ig\ca$.
	And hence, we have an uniform $LH_a^1$-bound in $h$ for all the difference quotients of first order in the open set $U_i$, which therefore converge weakly to the weak derivatives satisfying the estimate 
	\[
	\|\nabla_\ig^2 u\|_{LH_a^0(U_i)}^2\leq C\big( \|u_0\|_{H^1(M)}^2+\|f\|_{LH_a^0}^2 \big).
	\]
	Here we use $LH_a^0(U_i)$ to symbolize the corresponding function space on $U_i$. Finally, by summing over all coordinate charts, we obtain the desired result.
\end{proof}

As a direct consequence of Lemma \ref{lem:regularity2}, $u$ is smooth enough that we pointwisely have $\pa_t u =f+ \ca\De_\ig u-\cl[u]\in LH_a^0$. This, together with the fact $LH_a^0=P_a^0$, gives us an estimate for the time derivative of $u$:
\begin{Cor}\label{cor:regularity2}
	If $u\in HH_a$ is a weak solution of \eqref{equ-initial-general} with $f\in P_a^0$ and $u_0\in H^1(M)$, then $u\in P_a^1$ and
	\[
	\|u\|_{P_a^1}^2\leq C\big( \|u_0\|_{H^1(M)}^2+\|f\|_{P_a^0}^2 \big)
	\]
	provided $a>0$ is large enough in the sense of Proposition \ref{prop:existence in HH-a} and Lemma \ref{lem:regularity1}, where $C>0$ depends only on $\ca,\cl,\nabla_\ig\ca$ and the manifold $M$.
\end{Cor}

Now we are ready to give the $P_a^l$-version of the a priori energy estimate for parabolic equation \eqref{equ-initial-general}. 

\begin{proof}[Proof of Lemma \ref{lem:regularity3}]
	In this case, let us take a smooth cut-off function $\eta$ as in the proof of Lemma \ref{lem:regularity2} and choose $\phi=\varDelta_{-h}^{2l+k}(\eta^4\varDelta_h^{2l+k}u)$ for $h\neq0$ as the test function in \eqref{test}. Then we can proceed almost the same as before. The difference here lies in the bounds for the integrals, which come from the homogeneous term $f$ and the linear term $\cl[u]$. Indeed, we have
	\[
	\aligned
	|\inp{f}{\phi}_{LH_a^0}|&=\big|\inp{\varDelta_h^{2l}f}{\varDelta_{-h}^k(\eta^4\varDelta_h^{2l+k}u)}_{LH_a^0}\big|\leq \|f\|_{LH_a^{2l}}\|\varDelta_{-h}^k(\eta^4\varDelta_h^{2l+k}u)\|_{LH_a^0} \\
	&\leq \|f\|_{LH_a^{2l}}\|\eta^4\varDelta_h^{2l+k}u\|_{LH_a^k}
	\endaligned
	\]
	and, by recalling the hypothesis that $u\in P_a^l$ and the embedding $P_a^l\hookrightarrow LH_a^{2l}$, we have $\cl[u]\in LH_a^{2l}$ and hence
	\[
	\aligned
	|\inp{\cl[u]}{\phi}_{LH_a^0}|&=\big|\inp{\varDelta_h^{2l}\cl[u]}{\varDelta_{-h}^k(\eta^4\varDelta_h^{2l+k}u)}_{LH_a^0}\big|\leq\|\cl[u]\|_{LH_a^{2l}}\|\eta^4\varDelta_h^{2l+k}u\|_{LH_a^k}\\[0.2em]
	&\leq\|\cl[u]\|_{P_a^l}\|\eta^4\varDelta_h^{2l+k}u\|_{LH_a^k}\leq C_{\cl,l}\|u\|_{P_a^l}\|\eta^4\varDelta_h^{2l+k}u\|_{LH_a^k}.
	\endaligned
	\]
	Therefore, very similar to what have done in the proof of Lemma \ref{lem:regularity2}, we can deduce that
	\begin{\equ}\label{E6}
		\aligned
		\frac{3\de}4\|\eta^2\nabla_\ig(\varDelta_h^{2l+k}u)\|_{LH_a^0}^2&\leq C_\vr\|u\|_{LH_a^{2l+k}}^2+\vr\|\eta^2\nabla_\ig(\varDelta_h^{2l+k}u)\|_{LH_a^0}^2 +\frac12|\eta^2\varDelta_h^{2l+k}u_0|_2^2 \\
		&\qquad  + C_{\cl,l}\|u\|_{P_a^l}\|\eta^4\varDelta_h^{2l+k}u\|_{LH_a^k} + \|f\|_{LH_a^{2l}}\|\eta^4\varDelta_h^{2l+k}u\|_{LH_a^k}.
		\endaligned
	\end{\equ}
	And, by taking $\vr\leq\de/4$, $k=0$ and invoking the embedding $P_a^l\hookrightarrow LH_a^{2l}$, we first find that
	\[
	\|\eta^2\nabla_\ig(\varDelta_h^{2l}u)\|_{LH_a^0}^2\leq C\big( \|u_0\|_{H^{2l}(M)}^2+\|u\|_{P_a^l}^2+\|f\|_{P_a^l}^2 \big)
	\]
	where $C>0$ is independent of $h$. Hence, by taking to the limit $h\to0$ and summing over all coordinate charts, we obtain
	\begin{\equ}\label{u-LHa2l+1}
		\|u\|_{LH_a^{2l+1}}^2\leq C\big( \|u_0\|_{H^{2l}(M)}^2+\|u\|_{P_a^l}^2+\|f\|_{P_a^l}^2 \big)
	\end{\equ}
	with $C>0$ depends only on $l,\ca,\cl,\nabla_\ig\ca$ and the manifold $M$. Next, let us take $k=1$ in \eqref{E6} and use Young's inequality to obtain
	\[
	\|\eta^2\nabla_\ig(\varDelta_h^{2l+1}u)\|_{LH_a^0}^2\leq C \big( \|u_0\|_{H^{2l+1}(M)}^2+\|u\|_{LH_a^{2l+1}}^2 + \|u\|_{P_a^l}^2 +\|f\|_{LH_a^{2l}}^2 \big).
	\]
	Once again, by taking to the limit $h\to0$ and summing over all coordinate charts, we soon get
	\begin{\equ}\label{u-LHa2l+2}
		\aligned
		\|u\|_{LH_a^{2l+2}}&\leq C \big( \|u_0\|_{H^{2l+1}(M)}^2+\|u\|_{LH_a^{2l+1}}^2 + \|u\|_{P_a^l}^2 +\|f\|_{LH_a^{2l}}^2 \big) \\
		&\leq C \big( \|u_0\|_{H^{2l+1}(M)}^2+ \|u\|_{P_a^l}^2 +\|f\|_{LH_a^{2l}}^2 \big)
		\endaligned
	\end{\equ}
	where the last inequality follows from \eqref{u-LHa2l+1}. 
	
	So much for the spatial regularity, and we shall turn to consider the time derivatives. Notice that we have assumed $u\in P_a^l$ is a solution to $\pa_t u=f+\ca\De_\ig u -\cl[u]$, thus we have $\pa_t^su$ exists almost everywhere in $LH_a^{2(l-s)}$ for $s=0,\dots,l$ and
	\[
	\aligned
	\|\pa_t^{s+1}u\|_{LH_a^{2(l-s)}}^2&\leq C_s\Big( \|\pa_t^s f\|_{LH_a^{2(l-s)}}^2 + \|\pa_t^s\cl[u]\|_{LH_a^{2(l-s)}}^2 +\sum_{j\leq s}\|\pa_t^j u\|_{LH_a^{2(l+1-j)}}^2 \Big) \\
	&\leq C_s\Big( \|f\|_{P_a^l}^2 + \|u\|_{P_a^l}^2 +\sum_{j\leq s}\|\pa_t^j u\|_{LH_a^{2(l+1-j)}}^2 \Big)
	\endaligned
	\]
	with $C_s>0$ depends on the $L^\infty$-norms of $\ca,\pa_t\ca,\dots,\pa_t^s\ca$. Now, starting from Corollary \ref{cor:regularity2}, \eqref{u-LHa2l+2}, $s=0$ and iterating up to $s=l$, we obtain the desired conclusion.
\end{proof}

\subsection{Technical results}\label{A. Technical results}

In this section, let us collect some basic properties of the map $\msf$ defined in \eqref{the-map-msf}, which are crucial in the proof of our short time existence result.
	
	\begin{Lem}\label{lem:msf-C1}
		For $l\in\N$, $l>(m+6)/2$ and $T>0$ small, let $u\in P_{+,s}^{l+1}(T)$ be such that $\msf$ is well-defined. It holds that $\msf[u]\in H^{2l+1}(M)\times P^l(T)$. Moreover, for a $\cq$-admissible neighborhood $U$ of $u$, the map $\msf$ is of class $C^1$ on $U$.
	\end{Lem}
	\begin{proof}
		Note that the first component of $\msf$, i.e., the map $v\mapsto v(\cdot,0)$ is linear and bounded from $P^{l+1}(T)$ to $H^{2l+1}(M)$. As a result, it falls within the class $C^1$. Additionally, the map $v\mapsto \pa_t v$ is also linear and bounded, by taking functions from $P^{l+1}(T)$ to $P^l(T)$. Therefore, it is likewise of class $C^1$. Now, it remains to show that the following two maps 
		\[
		\cq_1[v]:=c_m v^{-\frac4{m-2}}\De_{\ig} v-\scal_{\ig} v^{\frac{m-6}{m-2}}
		\]
		and
		\[ 
		\cq_2[v]:=\Big(\int_M v L_{\ig}v \,d\vol_{\ig}\Big)|\psi_v|_{\ig}^2 v ^{-\frac m{m-2}}
		\]
		belong to $C^1(U,P^l(T))$ for some neighborhood  $U\subset P_{+,s}^{l+1}(T)$ of $u$.
		
		The analysis for the $C^1$-property of the operator $\cq_1$ can be found in \cite[Lemma 2.5]{MM}, where we only need to assume that $l>m/4$ to obtain $\cq_1\in C^1(P_+^{l+1}(T), P^l(T))$. Hence, in what follows, we will focus on the operator $\cq_2$, for which we have to pay more attention on the Dirac eigenspinor. To proceed, let $U\subset P_{+,s}^{l+1}(T)$ be an open and bounded neighborhood of the given function $u$ so that the simple eigenvalue $t\mapsto\lm_v(t)$ is well-defined for all $v\in U$, then our first step is to prove that $\cq_2\in C(U,P^{l+1}(T))$. 
		
		Let us fix $v\in U$. In order to show that $\cq_2[v]\in P^{l+1}(T)$, we need to establish some basic estimates for derivatives of the eigenspinor $\psi_v$. For this purpose, let us go back to Lemma \ref{lem:Ck-perturbed-eigenvalue}, where the derivation of the eigenpair $(\lm_v(t),\psi_v(t))$ is presented, and we shall firstly consider the growth of $\lm_v(t)$ with respect to $t$. 
		
		In general cases, let $\psi(t)$ be an eigenspinor of $v(\cdot,t)^{-\frac2{m-2}}D_\ig$ with eigenvalue $\lm(t)$ so that $\lm(0)\in[-n_0,n_0]\setminus\{0\}$. Using the expression of $\lm'$ in Lemma \ref{lem:Ck-perturbed-eigenvalue}, it follows from the $P^{l+1}$-boundedness of $v$ and the parabolic Sobolev embedding $P^{l+1}(T)\hookrightarrow C^1(M\times[0,T])$ for $l>(m+2)/4$ that there exists a positive constant $C>0$ (independent of $v\in U$) such that
		\[
		|\lm'(t)|\leq C|\lm(t)| \quad \text{for } t\in[0,T].
		\] 
		The solution of this differential inequality leads to the following bound for the growth of the perturbed eigenvalues:
		\[
		|\lm(t)-\lm(0)|\leq n_0\big(e^{Ct}-1\big), \quad t\in[0,T].
		\]
		
		Let us write $\lm_{0,1}<\lm_{0,2}<\dots<\lm_{0,j_0}$ for all the simple eigenvalues of $v(\cdot,0)^{-\frac2{m-2}}D_\ig$ that belong to $[-n_0,n_0]\setminus\{0\}$. Let $d_1,\dots,d_{j_0}>0$ be such that the intervals $I_j=[\lm_{0,j}-d_j,\lm_{0,j}+d_j]$ for $j=1,\dots,j_0$ are disjoint. Set $\lm_{0,-}$ for the biggest eigenvalue in $(-\infty,-n_0)$ and $\lm_{0,+}$ for the smallest eigenvalue in $(n_0,+\infty)$. By taking $d_1$ and $d_{j_0}$ smaller if necessary, we can assume that $\lm_{0,-}\not\in I_1$ and $\lm_{0,+}\not\in I_{j_0}$ (see Fig. \ref{F1}).
		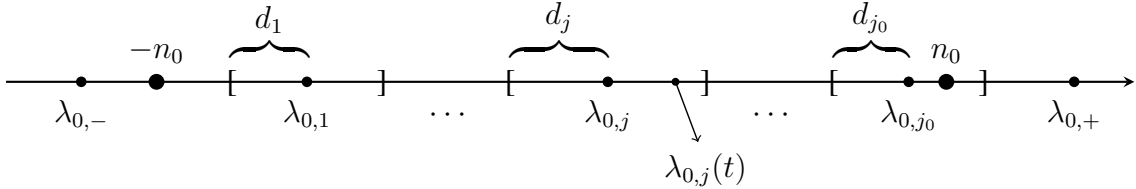
\begin{figure}[ht]
			\centering
			\begin{tikzpicture}[scale=1]
				\draw[thick,->,>=stealth,line width=0.15ex] (0,0) -- (15,0); 
				\coordinate [label=-90:$ \lm_{0,-}$] (a) at (1,-0.1); 
				\draw (a);
				\coordinate [label=90:$-n_0$] (b) at (2,0.1); 
				\draw (b);
				\coordinate [label=-90:$ \lm_{0,1}$] (c) at (4,-0.1); 
				\draw (c);
				\coordinate [label=-90:$\cdots$] (zz) at (5.9,-0.2); 
				\draw (zz);
				\coordinate [label=-90:$ \lm_{0,j}$] (d) at (8,-0.1); 
				\draw (d);
				\coordinate [label=-90:$\cdots$] (zzz) at (10.2,-0.2); 
				\draw (zzz);
				\coordinate [label=-90:$ \lm_{0,j_0}$] (e) at (12,-0.1); 
				\draw (e);
				\coordinate [label=90:$n_0$] (f) at (12.5,0.1); 
				\draw (f);
				\coordinate [label=-90:$ \lm_{0,+}$] (g) at (14.2,-0.1); 
				\draw (g);
				\foreach \p in {a,c,d,e,g} \fill (\p)++(0,0.1) circle (2pt);
				\foreach \p in {b,f} \fill (\p)++(0,-0.1) circle (3pt);
				\coordinate [label=90:$\text{[}$] (bb1) at (6.7,-0.34); 
				\draw (bb1);
				\coordinate [label=90:$\text{]}$] (bb2) at (9.3,-0.34); 
				\draw (bb2);
				\coordinate [label=90:$\text{[}$] (b1) at (3,-0.34); 
				\draw (b1);
				\coordinate [label=90:$\text{]}$] (b2) at (5,-0.34); 
				\draw (b2);
				\coordinate [label=90:$\text{[}$] (bbb1) at (11,-0.34); 
				\draw (bbb1);
				\coordinate [label=90:$\text{]}$] (bbb2) at (13,-0.34); 
				\draw (bbb2);
				\coordinate [label=90:$\overbrace{\hspace{2.5em}}$] (ob1) at (3.5,0.15); 
				\draw (ob1);
				\coordinate [label=90:$d_1$] (d1) at (3.5,0.5); 
				\draw (d1);
				\coordinate [label=90:$\overbrace{\hspace{3.1em}}$] (ob2) at (7.35,0.15); 
				\draw (ob2);
				\coordinate [label=90:$d_j$] (dj) at (7.35,0.5); 
				\draw (dj);
				\coordinate [label=90:$\overbrace{\hspace{2.4em}}$] (ob3) at (11.5,0.15); 
				\draw (ob3);
				\coordinate [label=90:$d_{j_0}$] (dj0) at (11.5,0.5); 
				\draw (dj0);
				\coordinate [label=-90:$ \lm_{0,j}(t)$] (h) at (9.3,-0.8); 
				\draw (h);
				\foreach \p in {h} \fill (\p)++(-0.4,0.8) circle (1.5pt);
				\draw[thick,->,line width=0.1ex] (8.9,0) -- (9.2,-0.8); 
			\end{tikzpicture}
			\caption{\it Simple eigenvalues of $v(\cdot,t)^{-\frac2{m-2}}D_\ig$ that belong to $(\lm_{0,-},\lm_{0,+})\setminus\{0\}$}\label{F1}
		\end{figure}
		
		To ensure that, for each $j=1,\dots,j_0$, the perturbed eigenvalue $\lm_{0,j}(t)$ belongs $I_j$, we only need to narrow the time interval $[0,T]$ so that
		\[
		|\lm_{0,j}(t)-\lm_{0,j}|\leq n_0\big( e^{CT}-1 \big)\leq d_j \quad \text{for all } j=1,\dots,j_0.
		\]
		And, by \cite[Proposition 7.1]{Bar1996}, we can also ensure that none of the perturbations of the eigenvalues that initially belonged to $\R\setminus[-n_0,n_0]$ coincide with $\lm_{0,1}(t),\dots,\lm_{0,j_0}(t)$. Therefore,  once the initial datum $\lm_v(0)\in\{\lm_{0,1},\dots,\lm_{0,j_0}\}\cap[-n_0,n_0]\setminus\{0\}$ is given for a function $v\in U$, we can find a constant $C_*>0$ (independent of $v$ and $t$) such that 
		\begin{\equ}\label{inverse-bdd}
		\big\|\big(v(\cdot,t)^{-\frac2{m-2}}D_\ig-\lm_v(t)\big)^{-1}\circ(I-P_{\lm_v(t)})\big\|_{L^2\to L^2} \leq C_*.
		\end{\equ}
		Hence, from Lemma \ref{lem:Ck-perturbed-eigenvalue}, we find $\psi_v'(\cdot)$ exists and is $L^2$-uniformly bounded with respect to all $v\in U$. Actually the above proof gives us more,  namely we can conclude from the elliptic regularity of Dirac operator that the map $t\mapsto \psi_v(t)$ induces a $C^1$-curve in the space of spinors.
		
		\medskip
		
		Taking into account a function $v\in U\subset P^{l+1}(T)$ has weak time-derivatives up to order $l+1$, in order to prove that $\cq_2\in C(U,P^{l+1}(T))$, we need to show that the map
		\[
		v\mapsto \pa_t^j\cq_2[v]
		\]
		belongs to $C(U,L^2([0,T],H^k(M)))$ whenever $j=0,1,\dots,l+1$ and $2j+k\leq 2(l+1)$.
		
		We can assume from now on that $k_l\geq1$ is the biggest integer so that $P^{l+1}(T)\hookrightarrow C^{k_l}(M\times[0,T])$. By the parabolic Sobolev embeddings (see Proposition \ref{Parabolic-Sobolev-Embeddings}), it follows easily that
		\[
		l=\Big[\frac{m+2}4\Big]+k_l
		\]
		where $[r]$ stands for the largest integer that doesn't exceed $r\in(0,+\infty)$. To estimate the time-derivatives of $\cq_2[v]$, we have to show that the corresponding derivatives of  $\psi_v$ exist and possess some necessary regularity. In this direction, we can first apply Lemma \ref{lem:Ck-perturbed-eigenvalue}  to obtain the time-derivatives of $\psi_v$ up to order $k_l$ which are all continuous in $t$. To go further, it can be seen from an induction argument that, up to certain constants, $\pa_t^j\cq_2[v]$ is a summation of functions in the form of
		\begin{\equ}\label{pa-j-cq-2}
			\Big(\int_M (\pa_t^{\al_1}v)L_\ig(\pa_t^{\al_2}v)d\vol_{\ig}\Big) \cdot \real \big(\pa_t^{\bt_1}\psi_v,\pa_t^{\bt_2}\psi_v\big)_\ig \cdot (v ^{-\ga_0})(\pa_tv)^{\ga_1}\cdots(\pa_t^{n_j}v)^{\ga_{n_j}}
		\end{\equ}
		where $\al_1,\al_2,\bt_1,\bt_2,\ga_1,\dots,\ga_{n_j}$ are non-negative integers 
		with $\al_1+\al_2+\bt_1+\bt_2+\sum_{i=1}^{n_j}i\ga_i=j$ and $\ga_0\geq0$. And it is easy to see that, for each $j\in\{k_l+1,\dots,l+1\}$, there exists at most one element in $\{\al_1,\al_2,\bt_1,\bt_2,n_j\}$ that is greater than $\frac{l+1}2$. And we remark here that, when $l$ is suitably large (say at least $l\geq(m+4)/2$), there holds $\frac{l+1}2\leq k_l$.
		
		If $\max\{\al_1,\al_2\}>k_l$, then we can assume without loss of generality that $\al_1>k_l$ (since the conformal Laplace operator $L_\ig$ is self-adjoint, we can always interchange the places of $\al_1$ and $\al_2$). And in this case, by using H\"older's inequality, we find that the function in \eqref{pa-j-cq-2} has the same integrability  as $\pa_t^{\al_1}v$ since the remaining factors are all $L^\infty$-bounded. In particular, we have
		\[
		\int_0^T\Big|\int_M (\pa_t^{\al_1}v)L_\ig(\pa_t^{\al_2}v)d\vol_{\ig}\Big|^{q_{\al_1}}dt
		\leq C(q_{\al_1}, \|v\|_{C^{\al_2+2}(M\times[0,T])})\|\pa_t^{\al_1}v\|_{L^{q_{\al_1}}(M\times[0,T])}^{q_{\al_1}}
		\]
		where 
		\[
		q_{\al_1}=\begin{cases}
			\frac{m+2}2(\frac{m-2}4-l+\al_1)^{-1} & \text{if } \frac{m-2}4-l+\al_1>0, \\
			\text{any finite number} & \text{if } \frac{m-2}4-l+\al_1=0,
		\end{cases}
		\]
		and the coefficient $C(q_{\al_1}, \|v\|_{C^{\al_2+2}(M\times[0,T])})>0$ depends only on $q_{\al_1}$ and $\|v\|_{C^{\al_2+2}(M\times[0,T])}$. Notice that $\al_1\geq k_l+1$ and $\al_1+\al_2\leq j\leq l+1$, we can deduce $\al_2\leq [\frac{m+2}4]$. Since we have assumed $l\geq (m+6)/2$, we find that $k_l\geq [\frac{m+2}4]+2$. This implies $\al_2+2\leq k_l$ and, therefore, $\|v\|_{C^{\al_2+2}(M\times[0,T])}<+\infty$. 
		
		If $n_j>k_l$, then we get the corresponding exponent $\ga_{n_j}=1$. And similar to the previous case, we once again conclude that the function in \eqref{pa-j-cq-2} has the same integrability  as $\pa_t^{n_j}v$ since the remaining factors are all $L^\infty$-bounded. It remains to consider the case $\max\{\bt_1,\bt_2\}>k_l$, and we can also assume without loss of generality that $\bt_1>k_l$. In order to study the integrability of $\pa_t^{\bt_1}\psi_v$, let us go back to \eqref{Phi} and consider the identity \eqref{Ck-perturbed-identity} because this is where $\psi_v$ and its derivatives originate. Clearly, the derivatives of $\lm_v$ will be also involved as  was pointed out in Lemma \ref{lem:Ck-perturbed-eigenvalue}.
		
		Suppose that the time-derivatives of $\lm_v$ and $\psi_v$ exist up to order $s\in\N\cup\{0\}$ and, in particular, there holds
		\begin{\equ}\label{s-th derivative}
			\lm_v^{(s)}\in L^2([0,T]) \quad \text{and} \quad \pa_t^{s}\psi_v\in L^2([0,T], H^{2(l+1)-2s+1}(M,\mbs(M)))
		\end{\equ}
		where $\lm_v^{(s)}$ stands for the $s$-th derivative of $\lm_v$, and $\lm_v^{(0)}=\lm_v$ To consider the $(s+1)$-th derivatives, we need to look at the mapping  
		\[
		t\mapsto \big( \nabla_{(\lm,\psi)}\Phi(v(\cdot,t),\lm_v(t),\psi_v(t)) \big)^{-1}[\ka_{s+1}(t),\phi_{s+1}(t)],
		\]
		where $\ka_{s+1}$ is a summation of functions in the form of 
		\[
		\real\int_M (v^{\mu_0})(\pa_tv)^{\mu_1}\cdots(\pa_t^{r_{s+1}})^{\mu_{r_{s+1}}}\big( \pa_t^{\nu_1}\psi_v,\pa_t^{\nu_2}\psi_v \big)_\ig d\vol_{\ig}
		\]
		with $\nu_1+\nu_2+\sum_{i=1}^{r_{s+1}}i\mu_i=s+1$, $\max\{\nu_1,\nu_2\}\leq s$ and $\mu_0\in\R$, and $\phi_{s+1}$ is a summation of spinors in the form of
		\[
		\lm_v^{(\sa)}\cdot(v^{\rho_0})(\pa_tv)^{\rho_1}\cdots(\pa_t^{d_{s+1}})^{\rho_{d_{s+1}}}\pa_t^\om\psi_v
		\]
		and 
		\[
		(v^{\chi_0})(\pa_tv)^{\chi_1}\cdots(\pa_t^{b_{s+1}}v)^{\chi_{b_{s+1}}}D_\ig \pa_t^{\theta}\psi_v
		\]
		with $\sa+\om+\sum_{i=1}^{d_{s+1}}i\rho_i=s+1$, $\max\{\sa,\om\}\leq s$, $\theta+\sum_{i=1}^{b_{s+1}}i\chi_i=s+1$, $\theta\leq s$ and $\rho_0,\chi_0\in\R$. Then, once \eqref{s-th derivative} is satisfied, we can deduce that $\lm_v^{(s+1)}$ and $\pa_t^{s+1}\psi_v$ exist and satisfy
		\[
		\lm_v^{(s+1)}=-(\phi_{s+1},\psi_v)_2^{v}\in L^2([0,T]) 
		\]
		and
		\[
		\aligned
		\pa_t^{s+1}\psi_v&=\frac{\ka_{s+1}}2\psi_v+\big( v^{-\frac2{m-2}}D_\ig-\lm_v \big)^{-1}\circ(I-P_{\lm_v})\phi_{s+1} \\[0.3em]
		&\in L^2([0,T], H^{2(l+1)-2(s+1)+1}(M,\mbs(M))).
		\endaligned
		\]
		Notice that, by using the elliptic regularity of the Dirac operator, \eqref{s-th derivative} is valid for $s=0$. Therefore, by an induction argument, the above analysis indicates that \eqref{s-th derivative} is valid for all $s=0,1,\dots, l+1$.

		Now, let us turn back to \eqref{pa-j-cq-2}. By using H\"older's inequality, an immediate consequence of  the above arguments is that $\pa_t^j\cq_2[v]\in L^2([0,T],L^2(M))\simeq L^2(M\times[0,T])$ and the map $v\mapsto \pa_t^j\cq_2[v]$ belongs to $C(U,L^2(M\times[0,T]))$ for $j=0,1,\dots,l+1$. We also notice that, the space or  mixed space-time derivatives $\pa_t^j\nabla^k \cq_2[v]$ with $2j+k\leq 2(l+1)$ can be treated similarly, by observing that the function $\pa_t^r\nabla^p v$ has the same integrability of $\pa_t^{r+p/2}v$ from the view point of the embeddings in Proposition \ref{Parabolic-Sobolev-Embeddings}. 
		%
		This will suggest that $\cq_2\in C(U, P^{l+1}(T))$ as was desired.
		
		To see that the Fr\'echet derivative $D\cq_2$ exits and belongs to $C(U,\msl(P^{l+1}(T),P^{l+1}(T)))$, we shall first look at the Gateaux derivative 
		\[
		(v,w)\mapsto d\cq_2(v,w):=\frac{d}{d\tau}\cq_2[v+\tau w]\Big|_{\tau=0}
		\]
		for $v\in U$ and $w\in P^{l+1}(T)$. A direct computation, together with Lemma \ref{lem:Ck-perturbed-eigenvalue}, shows that
		\begin{\equ}\label{Dcq2}
		\aligned
		d\cq_2(v,w)&=2\Big(\int_M  v L_\ig w\, d\vol_\ig\Big)|\psi_v|_\ig^2v^{-\frac m{m-2}}+2\Big(\int_M  v L_\ig v \, d\vol_\ig\Big)\real(\psi_v,\phi_{v,w})_\ig v^{-\frac m{m-2}} \\
		&\qquad -\frac m{m-2}\Big(\int_M  v L_\ig v \, d\vol_\ig\Big)|\psi_v|_\ig^2v^{-\frac {2m-2}{m-2}}w
		\endaligned
		\end{\equ}
		where
		\[
		\phi_{v,w}(t)=\frac{\ka_{v,w}(t)}2\psi_v(t)-\frac{2\lm_v(t)}{m-2}\big( v(\cdot,t)^{-\frac2{m-2}}D_\ig-\lm_v(t) \big)^{-1}\circ (I-P_{\lm_v(t)})\big( v(\cdot,t)^{-1}w(\cdot,t)\psi_v(t) \big)
		\]
		with
		\[
		\ka_{v,w}(t)=-\frac2{m-2}\int_Mv(\cdot,t)^{\frac{4-m}{m-2}}w(\cdot,t)|\psi_v(t)|_\ig^2d\vol_\ig.
		\]
		Then it is easy to see that $d\cq_2(v,w)$ is linear in $w$. And it is clear, since $w\in P^{l+1}(T)$, that the same estimates used to show the continuity of $\cq_2$ can be repeated here to prove $d\cq_2(v,\cdot)\in \msl(P^{l+1}(T),P^{l+1}(T))$. Moreover, an analogous reasoning also applies to the continuity of $v\mapsto d\cq_2(v,\cdot)$. Therefore, the Fr\'echet derivative of $\cq_2$ exits with $D\cq_2[v]=d\cq_2(v,\cdot)$, for all  $v\in U$, which completes the proof.
	\end{proof}
	
By computing the Fr\'echet derivative of $\msf$ at $v\in U\subset P_{+,s}^{l+1}(T)$, we find that
\[
D\msf[v](w)=\frac{d}{d\tau}\msf[v+\tau w]\Big|_{\tau=0}=\big(w_0,  \pa_tw - c_mv^{-\frac4{m-2}}\De_\ig w + \cl_v[w] \big)
\]
for $w\in P^{l+1}(T)$, where $\cl_v \in\msl(P^{l+1}(T),P^{l+1}(T))$ is given by 
\begin{\equ}\label{cl-v}
\aligned
\cl_{v}[w] &= \frac{4c_m}{m-2}v^{-\frac{m+2}{m-2}}w\De_{\ig} v-\frac{m-6}{m-2}\scal_{\ig}v^{-\frac{4}{m-2}}w \\
&\qquad +2\vartheta\cdot\Big(\int_M w L_{\ig}v \,d\vol_{\ig}\Big)|\psi_{v}|_{\ig}^2 v^{-\frac m{m-2}} \\
&\qquad +2\vartheta\cdot\Big(\int_M v L_{\ig}v \,d\vol_{\ig}\Big)\real(\psi_{v},\phi_{v,w})_{\ig} v^{-\frac m{m-2}} \\
&\qquad -\frac{m}{m-2}\vartheta\cdot\Big(\int_M vL_{\ig}v \,d\vol_{\ig}\Big)|\psi_{v}|_{\ig}^2 v^{-\frac {2m-2}{m-2}} w
\endaligned
\end{\equ}
with
\[
\aligned
\phi_{v,w}(t)&=\frac{\ka_{v,w}(t)}2\psi_{v}(t) \\
&\qquad -\frac{2\lm_{v}(t)}{m-2}\big( v(\cdot,t)^{-\frac2{m-2}}D_\ig-\lm_{v}(t) \big)^{-1}\circ(I-P_{\lm_{v}(t)})\big( v(\cdot,t)^{-1} w(\cdot,t)\psi_{v}(t) \big)
\endaligned
\]
and
\[
\ka_{v,w}(t)=-\frac2{m-2}\int_M v(\cdot,t)^{\frac{4-m}{m-2}}w(\cdot,t)|\psi_{v}(t)|_\ig^2d\vol_\ig.
\]
The following result is a direct consequence of Lemma \ref{lem:msf-C1}.

\begin{Cor}\label{lem:A1-A3}
If $\msf$ is $C^1$ on $U\subset P_{+,s}^{l+1}(T)$, for some $l\in\N$, and $v\in U\cap C^\infty(M\times[0,T])$, then $\cl_v$ satisfies 
\begin{itemize}
	\item[$(1)$] $|\cl_v[w](\cdot,t)|_2\leq C \|w(\cdot,t)\|_{H^1(M)}$ for some constant $C>0$;
	
	\item[$(2)$] $\cl_v[\al w](\cdot,t)=\al(t)\cl_v[w](\cdot,t)$ for any smooth function $\al:[0,T]\to\R$;
	
	\item[$(3)$] $\cl_v\in\msl(P^j(T),P^j(T))$ for all $j\in\N$.
\end{itemize} 
\end{Cor}
\begin{proof}
Clearly, the validity of $(2)$ follows directly from the linearity of $\cl_v$. And, by repeating the arguments of Lemma \ref{lem:msf-C1}, we can easily check $(3)$ is satisfied as $v$ is smooth. So we only need to confirm that $\cl_v$ satisfies $(1)$.

In fact, using the smoothness and positiveness of $v$ on $M\times[0,T]$, the $L^2$-norm of $\cl_v[w](\cdot,t)$ can be estimated via \eqref{cl-v} as 
\[
\aligned
|\cl_v[w](\cdot,t)|_2&\leq C \Big( |w(\cdot,t)|_2 + \Big|\int_M w L_{\ig}v \,d\vol_{\ig}\Big| +|\phi_{v,w}(t)|_2 \Big)  \\
&\leq C\big(|w(\cdot,t)|_2 + \|w(\cdot,t)\|_{H^1(M)}+|\phi_{v,w}(t)|_2\big)
\endaligned
\]
where we have used 
\[
\Big|\int_M w L_{\ig}v \,d\vol_{\ig}\Big| =\Big|\int_M \big( c_m \nabla_\ig w\cdot\nabla_\ig v + \scal_{\ig}vw \big) d\vol_{\ig}\Big| \leq C\|w(\cdot,t)\|_{H^1(M)}
\]
with $C>0$ independent of $w$ and $t$.
Notice that, by \eqref{inverse-bdd}, we have $|\phi_{v,w}(t)|_2\leq C|w(\cdot,t)|_2$. Therefore, we obtain $(1)$ as was asserted. 
\end{proof}
	
%

\end{document}